\documentclass{article}
\usepackage{graphicx} 
\usepackage{amsmath}
\usepackage{amsthm}
\usepackage{ amssymb }
\usepackage[utf8]{inputenc}
\usepackage{amsfonts}
\usepackage{placeins}
\usepackage{tikz}
\usepackage{tikz-cd}
\usetikzlibrary{tqft}
\usepackage{subcaption}
\usepackage{amsfonts,graphicx}
\usepackage{caption}
\usepackage{thmtools}
\usepackage{mathtools}
\usepackage{physics}
\usepackage{ textcomp }
\usepackage{bbold}

\usepackage[shortlabels]{enumitem}
\usepackage[maxbibnames=9]{biblatex}

\usepackage{tocloft}
\usepackage{hyperref}
\hypersetup{colorlinks,allcolors=black}
\newcommand{\Sp}{\mathbb{S}}
\newcommand{\To}{\mathbb{T}}

\newtheorem{theo}{Theorem}[section]
\newtheorem{defi}[theo]{Definition}
\newtheorem{lemm}[theo]{Lemma}
\newtheorem{pro}[theo]{Proposition}
\newtheorem{exa}[theo]{Example}
\newtheorem{rem}[theo]{Remark}
\newtheorem{col}[theo]{Corollary}

\numberwithin{equation}{section}
\title{Symplectic embeddings of toric domains with boundary a lens space}
\author{Jonathan Trejos}
\date{}

\begin{document}

\maketitle
\begin{abstract}
 We give a combinatorial description of the embedded contact complex (ECC) of a certain family of contact toric lens spaces that we call concave lens spaces. We also define a notion of a concave toric domain that generalizes the usual concave toric domain in a way that possesses a singularity point and has a boundary a lens space. After desingularization these toric domains include the unitary cotangent bundle of $\mathbb{S}^2$ and the unitary cotangent bundle of $\mathbb{R}P^2$. We use the combinatorial expression of the ECC to compute the ECH capacities of these toric domains. Furthermore, for certain concave toric domains we describe a packing of symplectic manifolds that recovers their ECH capacities.
\end{abstract}

\setcounter{tocdepth}{3}
\tableofcontents

\section{Introduction}
Let $(X,\omega)$ be a symplectic four-manifold, possibly with boundary or corners, non-compact, maybe disconnected. Its $\text{ECH}$ capacities are a sequence of real numbers 
\begin{equation}
    0=c_0(X,\omega)\leq c_1(X,\omega)\leq c_2(X,\omega)\leq \cdots\leq \infty
\end{equation}
The $\text{ECH}$ capacities were introduced in \cite{hutchings2011quantitative}, see also \cite{hutchings2014lecture}. We give more detail about these definitions in Section \ref{ECHfoundations}.

The following are elementary properties of the ECH capacities:
\begin{enumerate}
    \item \textit{(Monotonicity)}  If there exists a symplectic embedding $(X,\omega)\xhookrightarrow{s} (X',\omega')$, then $c_k(X,\omega)\leq c_k(X',\omega')$ for all $k$.
    \item \textit{(Conformality)} If $r>0$ then
    $$c_k(X,r\omega)=r c_k(X,\omega)$$
    \item \textit{(Disjoint union)}
    $$c_k\left(\coprod_{i=i}^n (X_i,\omega_i)\right)=\max_{k_1+\cdots+k_n=k}\displaystyle\sum_{i=1}^n c_{k_i}(X_i,\omega_i)$$
    \item \textit{(Ellipsoid)} If $a,b>0$, define the ellipsoid
    $$E(a,b)=\left\{(z_1,z_2)\in \mathbb{C}^2:\displaystyle\frac{\pi |z_1|^2}{a}+\displaystyle\frac{\pi |z_2|^2}{b}\leq 1\right\}$$
    Then $c_k(E(a,b))=N(a,b)_k$, where $N(a,b)$ denotes the sequence of all nonnegative integer linear combinations of $a$ and $b$, arranged in nondecreasing order, indexed starting at $k=0$.
\end{enumerate}

A proof of these properties can be found in \cite{hutchings2014lecture}.

The computation of the $\text{ECH}$ capacities is not an easy task but several improvements have been done. An interesting family of symplectic four-manifolds for which valuable results were obtained is described as follows. If $\Omega$ is a domain in the first quadrant of the plane, define the toric domain 
$$X_\Omega=\{z\in \mathbb{C}^2:\pi (|z_1|^2,|z_2|^2)\in \Omega\}$$
For example, if $\Omega$ the triangle with vertices $(0,0), (a,0)$ and $(0,b)$, then $X_\Omega$ is the ellipsoid $E(a,b)$. 

The $\text{ECH}$ capacities of toric domains $X_\Omega$, when $\Omega$ is convex and does not touch the axes were computed in \cite{hutchings2011quantitative} theorem 1.11. The cases in which the region $\Omega$ touch the axis have received considerably more attention and they have special names:

\begin{defi}
    A \textit{convex toric domain} is a toric domain $X_\Omega$, where $\Omega$ is a closed region in the first quadrant bounded by the axes and a convex curve from $(a,0)$ to $(0,b)$, for $a$ and $b$ positive numbers. Similarly a concave domain $X_\Omega$ is a toric domain where $\Omega$ is a closed region in the first quadrant bounded by the axes and a concave curve from $(a,0)$ to $(0,b)$, for $a$ and $b$ positive number. 
\end{defi}

The $\text{ECH}$ capacties of concave domains were calculated \cite{Choi_2014}, as well as the capacities of the convex domains \cite{cristofaro2019symplectic}. In the present work we aim to generalize the notion of concave toric domains to consider symplectic manifolds with boundary a contact manifold diffeomorphic to a lens space. For these new concave domains, we compute combinatorial expressions to their $\text{ECH}$ capacities (see Theorem \ref{concavecapacities}). To properly define this generalization of concave toric domains we define (Section \ref{symplectictoricorb}) an orbifold with one singularity point that plays the role of $\mathbb{C}^2$ in the definition of concave domains. This singularity point can be removed in several ways. One interesting way to remove the singularity point is using the techniques for almost toric fibration introduced by Symington \cite{symington2002dimensions}. With the use of these techniques it is possible to recover as concave domains well known spaces as the unit cotangent bundle of $\Sp^2$ as well as the unit cotangent bundle of $\mathbb{R}P^2$ (see Examples \ref{S2} and \ref{RP2}) which recently some interesting properties where found by Ferreira and Ramos \cite{Ferreira_2022}.

We also want this present work to ground the basis to generalize some of the results obtained for the classic concave toric domains. We beging this project by generalizing the ball packing result from \cite{Choi_2014} in Section \ref{ballpacking}.

\subsection{Symplectic Toric Orbifolds}
\label{symplectictoricorb}

Given a pair of relatively primes positive integers $(n,m)$, our intention in this section is to define a symplectic orbifold $M(n,m)$ such that we can define an analogous to the toric domains introduced by Hutchings but with a boundary a lens space. The definition we use of lens space is as follows

\begin{defi}
\label{lensspace}
    Let $(t_1,t_2)$ be the coordinates on $\To^2=\Sp^1\times\Sp^1$, $x$ be the coordinate on $I=[0,1]$ and orient $I\times\To^2$ by the frame $\{\partial_{x},\partial_{t_1},\partial_{t_2}\}$. The \textit{lens space} $L(n,m)$ is the quotient of $I\times \To^2/\sim$ where $\sim$ collapses the integral curves of $\partial t_1$ on $\To^2\times \{0\}$ and collpases the integral curves of $n\partial_{t_1}-m\partial_{t_2}$.
\end{defi}

As we describe below this definition of a lens space is very useful to describde certain contact strutures over them. This definition can be found in \cite{lerman2000contact} and \cite{symington2002dimensions}.

    \subsubsection{Construction of the toric orbifold $M(n,m)$.}

    Fix a par of relative prime positive integer $(n,m)$ and consider the cone
    $$V_{n,m}=\{t_1(n,m)+t_2(0,1):t_1,t_2>0\}$$

    we call $\{t(0,1):t\geq 0\}$ and $\{t(n,m):t\geq 0\}$ the \textit{axis} of $V_{n,m}$

    Notice that $V_{n,m}\times \To^2$ can be regarded as a symplectic manifold with the $2$-form
    \begin{equation}
    \label{symplecticform}
    \omega=\displaystyle\frac{1}{2\pi }(\dd t_1\wedge \dd \theta_1+\dd t_2 \wedge \dd \theta_2) 
    \end{equation}
    where $(t_1,t_2)$ are the variables of $V_{n,m}$ and $(\theta_1,\theta_2)$ are the variables of $\To^2$.

    We are going to construct the manifold $M(n,m)$ as a quotient of $\overline{V}_{n,m}\times \To^2$. We say that $(t_1,t_2,\theta_1,\theta_2)\sim (t'_1,t'_2,\theta'_1,\theta'_2)$ as follows
    \begin{enumerate}[i]
    \item $(t_1,t_2,\theta_1,\theta_2)=(t'_1,t'_2,\theta'_1,\theta'_2)$.
    \item $t_1=t'_1=0$ and $(\theta_1,\theta_2+\theta)=(\theta'_1,\theta'_2)$ for some real $\theta$.
    \item $t_2=t'_2=0$ and $(\theta_1+m\theta,\theta_2-n\theta)=(\theta'_1,\theta'_2)$ for some real $\theta$.
    \end{enumerate}
    Then we define $M(n,m)=(\overline{V}_{n,m}\times \To^2)/\sim$ $M(n,m)$ is a symplectic toric orbifold with a canonical moment map $\pi:M(n,m)\rightarrow \bar{V}_{(p,q)}$.

    \begin{exa}
        Take $m=0$ and $n=1$ in this case it is easy to see that $M(1,0)$ is symplectomorphic to $\mathbb{C}^2$.
    \end{exa}

\subsubsection{Visible submanifolds of $M(n,m)$.}
\label{NotableSubs}
We begin this section by describing a family of lens spaces contained in $M(n,m)$.

\begin{lemm}
\label{exilens}
Let $a:[0,1]\rightarrow \bar{V}_{n,m}$ be a smooth curve  such that $a(0)$ lies in the ray $\{t(n,m):t>0\}$ and $a(1)$ lies in the ray $\{t(0,1):t>0\}$ then
\begin{enumerate}[i]
    \item $Y_a\vcentcolon=\pi^{-1}(a([0,1]))$ is diffeomorphic to the lens space $L(n,m)$.
    \item Suppose that $a\times a'>0$ then $Y_a$ is a contact manifold with contact form
    \begin{equation}
    \label{contacsstructure}
    \lambda_a=a_1 \dd t_1+a_2 \dd t_2     
    \end{equation}
    where $a=(a_1,a_2)$.
\end{enumerate}
\end{lemm}

\begin{proof}
Notice that $\pi^{-1}(a([0,1]))$ is exactly the quotient described in the definition \ref{lensspace}. Now consider the Liouville vector field
$$V=t_1\partial_{\theta_1}+t_2\partial_{\theta_2}$$
of the symplectic $2$-form \eqref{symplecticform}. The conditions over the function $a$ ensure us that $V$ and $Y_a$ are transversal. Then $\iota_V\omega$ is a contact structure equal to 
$$t_1\dd \theta_1+t_2 \dd\theta_2$$
Replacing $t_1$ and $t_2$ by $(a_1,a_2)$ give us the equation \eqref{contacsstructure}.
\end{proof}

Following \cite{symington2002dimensions} we called the contact $3$-manifold described above a \textit{visible lens space} of $M(n,m)$. 

\begin{defi}
    \label{domain}
    A domain $\Omega$ in $V_{n,m}$ is a bounded subset of $V_{n,m}$ for which there exists a curve $a:[0,1]\rightarrow V_{n,m}$ with $a(0)=(a_0 n a_0 m)$ and $a(1)=(0,a_1 1)$  such that $\partial \Omega=a[0,1]\cup \{t_1(p,q):0\leq t_1\leq a_1 \}\} \cup\{t_2(0,1):0\leq t_1\leq a_2\}$. If no explicit use of the curve $a$ is needed we will write $\partial^+ \Omega = a[0,1]$.
\end{defi}

\begin{defi} Let $\Omega$ be a domain in $V_{n,m}$. We called the symplectic orbifold $X_\Omega\vcentcolon=\pi^{-1}(\Omega)$ a \textit{toric domain} in $M(n,m)$. We say that $X_\Omega$ is a \textit{concave toric domain} if the complement of  $\Omega$ in $V_{n,m}$ is a convex set. The lens space $\pi^{-1}(\partial^+\Omega)$ with its respective contact struture is called a concave lens space. See figure \ref{Toric Concave Domain}.
\end{defi}

\begin{defi} Let $X_\Omega$ be a  concave toric domain in $M(n,m)$. As in Definition \ref{domain} let $a:[0,1]\rightarrow V_{n,m}$ be a curve such that  $a(0)=(a_0 n, a_0 m)$, $a(1)=(0,a_1 1)$ and  $\partial \Omega=a[0,1]\cup \{t_1(p,q):0\leq t_1\leq a_1 \}\} \cup\{t_2(0,1):0\leq t_1\leq a_2\}$. We say that $X_\Omega$ is a \textit{rational toric domain} if $a'$ is rational whenever is defined.     
\end{defi}

\begin{figure}
    \centering
    \tikzset{every picture/.style={line width=0.75pt}} 

\begin{tikzpicture}[x=0.75pt,y=0.75pt,yscale=-1,xscale=1]

\draw  [draw opacity=0] (166,4) -- (470,4) -- (470,227) -- (166,227) -- cycle ; \draw  [color={rgb, 255:red, 194; green, 194; blue, 194 }  ,draw opacity=1 ] (166,4) -- (166,227)(186,4) -- (186,227)(206,4) -- (206,227)(226,4) -- (226,227)(246,4) -- (246,227)(266,4) -- (266,227)(286,4) -- (286,227)(306,4) -- (306,227)(326,4) -- (326,227)(346,4) -- (346,227)(366,4) -- (366,227)(386,4) -- (386,227)(406,4) -- (406,227)(426,4) -- (426,227)(446,4) -- (446,227)(466,4) -- (466,227) ; \draw  [color={rgb, 255:red, 194; green, 194; blue, 194 }  ,draw opacity=1 ] (166,4) -- (470,4)(166,24) -- (470,24)(166,44) -- (470,44)(166,64) -- (470,64)(166,84) -- (470,84)(166,104) -- (470,104)(166,124) -- (470,124)(166,144) -- (470,144)(166,164) -- (470,164)(166,184) -- (470,184)(166,204) -- (470,204)(166,224) -- (470,224) ; \draw  [color={rgb, 255:red, 194; green, 194; blue, 194 }  ,draw opacity=1 ]  ;
\draw    (226,204) -- (226,46) ;
\draw [shift={(226,44)}, rotate = 90] [color={rgb, 255:red, 0; green, 0; blue, 0 }  ][line width=0.75]    (10.93,-3.29) .. controls (6.95,-1.4) and (3.31,-0.3) .. (0,0) .. controls (3.31,0.3) and (6.95,1.4) .. (10.93,3.29)   ;
\draw    (226,204) -- (464.34,45.11) ;
\draw [shift={(466,44)}, rotate = 146.31] [color={rgb, 255:red, 0; green, 0; blue, 0 }  ][line width=0.75]    (10.93,-3.29) .. controls (6.95,-1.4) and (3.31,-0.3) .. (0,0) .. controls (3.31,0.3) and (6.95,1.4) .. (10.93,3.29)   ;
\draw [color={rgb, 255:red, 50; green, 0; blue, 255 }  ,draw opacity=1 ]   (226,64) .. controls (242,133) and (265,155) .. (406,84) ;

\end{tikzpicture}
    \caption{Example of a Toric Concave Domain}
    \label{Toric Concave Domain}
\end{figure}
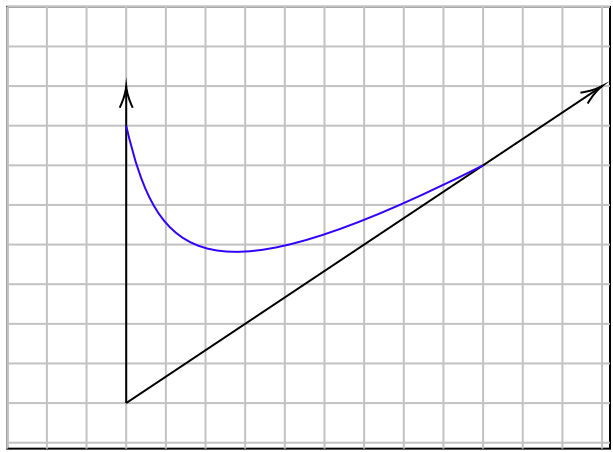
Let $a$ and $b$ be real positive numbers. Consider the domain $\Omega$ in $V_{n,m}$ defined as the convex hull of the vertices $(0,0), \frac{a}{m}(n,m)$ and $\frac{b}{m}(0,1)$. We denote $X_{\Omega}$ by $E_{n,m}(a,b)$ and we call it the \textit{ellipsoid with singularities} of periods $a$ and $b$. We define the ball with singularities as $B_{n,m}(a)\vcentcolon=E_{n,m}(a,b)$. To simplify notation we write $E_n (a,b)$ instead of $E_{n,1}(a,b)$. Notice that for $E_{n,m}(a,b)$ the elliptic orbit $e_1$ has period $a$ and $e_2$ has period $b$.

\subsection{ECH Capacities of toric domains in $M(n,m)$}

In this section we give a description of the ECH capacities of concave toric domains in $M(n,m)$. We begin by describing the capacities of ellipsoids with singularities. We give the proofs of the results of this section in Section \ref{ECH-of-Lens-Spaces}.

\subsubsection{ECH capacities of Ellipsoids with singularities.}

Given two real positive numbers $a,b$ and two positive integer numbers $n,m$, we define the sequence $N^{n,m}(a,b)$ as the sequence of numbers of the form $a k_1+bk_2$ such that there exist an integer $l$ such that $k_1+m k_2=ln$, the sequence $N^{n,m}(a,b)$ is organized by increasing order with repetitions. We denote the $k$-th number of the sequence $N^{n,m}(a,b)$ by $N_k^{n,m}(a,b)$. So the ECH capacities of ellipsoids with singularity are given by the following Lemma:

\begin{lemm}
\label{capacitiesellipsoids}
     The ECH capacities of an ellipsoid with singularities are given by the sequence defined above, i.e,

\begin{align}
\label{capLmn}
c_k(E_{n,m}(a,b))=N_k^{n,m}(a,b)    
\end{align}

where $a,b$ are positive real numbers and $n,m$ are positive integer numbers. 
\end{lemm}

We give a proof of Lemma \ref{capLmn} in Section \ref{ECH-of-Lens-Spaces}.

Of particular interest for us is the case when $m=1$. Lets describe this case in more detail. To simplify notation write $N^n(a,b):=N^{n,1}(a,b)$. Notice that by Lemma \ref{capacitiesellipsoids} the sequence $N^n(a,b)$ correspond to the ordered sequence of numbers of the form $ar+bs$ such that $r+s$ is a multiple of $n$ with repetitions. Then 
\begin{equation}
\label{caplen1p}
c_k(E_n(a,b))=N^n_k(a,b)    
\end{equation}
For any positive integer $k$. Using these results we can recover some of the result shown in \cite{Ferreira_2022} as described in the following examples.

\begin{exa}
\label{S2}
In the above example take $n=2$ and $a=b=1$ then $$c_k(B_2(1))=N^2_k(1,1)=(0,2,2,2,4,4,4,4,4,6,6,6,6,6,6,6, \dots)$$
Notice that $2\pi c(B_2(1))$ are exactly the capacities for $D^*\Sp^2$ calculated in \cite{Ferreira_2022}.  Using the observations of \cite{wu2015exotic} in Section 3. it can be shown that after a rational blown down $2\pi B_2(1)$ is symplectomorphic to $D^*\Sp^2$. This recovers part $(i)$ of Theorem 1.3 of \cite{Ferreira_2022}. 
It is easy to see that from the moment map that $\text{int } B(1)\xhookrightarrow{s} B_2(1)$ from which part (i) of theorem 1.1 of \cite{Ferreira_2022} follows. Futhermore, after a toric mutation we can see from the diagram that $P(1,1)\xhookrightarrow{s} B_2(1)$ which recorvers part $(iv)$ of Theorem 1.1 of \cite{Ferreira_2022}.
\end{exa}
\begin{exa}
\label{RP2}
Similarly take $n=4$ and $a=b=1$ then
    $$c_k(B_4(1))=N^4_k(1,1)=(0,4,4,4,4,4,8,8,8,8,8,8,8,8,8, \dots)$$
and notice that $\pi c(B_4(1))$ are exactly the capacities of $D^*\mathbb{R}P^2$. As in the previous case we can use the observations of \cite{wu2015exotic} in Section 3 to prove that $\pi B_4(1)$  after a rational blown down trade is symplectomorphic $D^*\mathbb{R}P^2$. This recovers part (ii) of Theorem 1.3 of \cite{Ferreira_2022}.

Using a toric mutations \cite{casals2022full} and noticing that the pair $(4,1)$ is of the form $(k^2,kl-1)$ we can see that $B(1)\xhookrightarrow{s} B_4(1)$ from which part (ii) of Theorem 1.1 of \cite{Ferreira_2022} follows.

\end{exa}

\subsubsection{ECH capacities of concave toric domains on $M(n,m)$}

Now we give a combinatorial description of the ECH capacities of concave toric domains on $M(n,m)$. This formula is similar to the one given in \cite[Sec.~ 1.6]{Choi_2014}.

\begin{defi}
    \label{(n,m)-concave path}
    A \textit{$(n,m)$-concave polygonal path} $P$ is a piecewise linear continuos path $\Lambda$ in the $xy$-plane with starting point a lattice point in the line $\{t(p,q):t\geq 0\}$ and end point at a lattice point the $y$-axis, and, $P$ is concave in the sense that it lies above any of the tangent lines at its smooth points. 
\end{defi}

\begin{defi}
\label{Counting points}
    If $\Lambda$ is a concave $(n,m)$-integral path, define $\mathcal{L}_{n,m}(\Lambda)$ to be the number of lattice points in the region bounded by $\Lambda$, the $y$-axis and the ray $\{t(n,m):t\geq 0\}$. Without counting the points in $\Lambda$.
\end{defi}

\begin{defi}
\label{omegalength}
        Let $X_\Omega$ be the concave toric domain in $M(n,m)$. Suppose that $\Lambda$ is a concave $(n,m)$-integral path, define the $\Omega$-\textbf{length} of $\Lambda$, as follows. For each edge $v$ let $p_v$ be a point in $\partial^+\Omega$ such that $\Omega$ is contained in the closed half-plane above the line through $p_v$ parallel to $v$. Then 
    $$l_{\Omega}(\Lambda)=\displaystyle\sum_{v\in \textit{Edges}(\Lambda)} v\times p_v$$
    Here $\times$ denote the cross product. Note that if $p_v$ is not unique then the value $v\times p_v$ does not depend on the choice of $p_v$. 
\end{defi}

\begin{theo}
    \label{concavecapacities}
    If $X_\Omega$ is any rational concave toric domain of $M(n,m)$, then its $ECH$ capacities are given by 
    $$c_k(X_\Omega)=\max\{l_\Omega(\Lambda):\mathcal{L}(\Lambda)=k\} $$
Here the maximum is over concave $(n,m)$-integral paths $\Lambda$.
\end{theo}

The proof of Theorem \ref{concavecapacities} is given at the end of Section \ref{ECH-of-Lens-Spaces}.

\subsection{Ball Packing}
\label{ballpacking}

In this section we specialized to the orbifolds $M(n,1)$. Here we want to describe the ECH capacities of a toric concave domain in $M(n,1)$ as the capacities of the disjoint union of balls with singularities.

\subsubsection{Weight Expansions}
\label{weightexpansions}

Suppose that $X_\Omega$ is a concave domain in $M(n,1)$. The \textit{weight expansion} of $\Omega$ is a finite unordered list of (possibly repeated) positive real numbers $w(\Omega)=\{a,a_1,\dots, a_n\}$ analogous to the weight expansion of a concave toric domain. Since we will need this definition we recall it here.

\textbf{Weight expansion of the usual concave toric domain in $\mathbb{C}^2$:} Suppose that $X_\Omega$ is a concave toric domain in $\mathbb{C}^2$. The \textit{weight expansion} $w(\Omega')$ is defined as follows (see \cite{Choi_2014} section 1.3). 

If $\Omega'$ is the triangle with vertices $(0,0),(a,0)$ and $(0,a)$ then $w(\Omega)=(a)$. 

Otherwise, let $a>0$ be the largest real number such that the triangle with vertices $(0,0),(a,0)$ and $(0,a)$ is contained in $\Omega$. Call this triangle $\Omega'_1$. The line $x+y=a$ intersect $\partial \Omega$ in a line segments from $(x_2,a-x_2)$  to $(x_3,a-x_3)$ with $x_2\leq x_3$. Let $\Omega''_2$ the portion of $\Omega$ above the line $x+y=a$ and to the left of the line $x=x_2$. By applying the translation $(x,y)\rightarrow (x,y-a)$ to $\Omega''_2$ and then multiplying by $\begin{pmatrix} 1 & 0 \\ 1 & 1 \end{pmatrix}$  we find a new domain $\Omega_2$ (which we interpret as the empty set if $x_2=0$). Let $\Omega''_3$ denote the portion of $\Omega'$ above the line $x+y=a$ and to the right of the line $x=x_3$. By first applying the translation $(x,y)\rightarrow (x-a,y)$ and then multiplying by $\begin{pmatrix} 1 & 1 \\ 0 & 1 \end{pmatrix}\in \text{SL}_2(\mathbb{Z})$. We now define 
$$w(\Omega')=w(\Omega''_1)\cup w(\Omega''_2)\cup w(\Omega''_3)$$

\textbf{Weight expansion of a rational concave toric domain in $M(n,1)$:} To define the weight expansion of a rational concave  toric domain $X_\Omega$ in $M(n,1)$ we proceed as follows. If $\Omega$ is the triangle with vertices $(0,0),(na,a)$ and $(0,a)$ then $w(\Omega)=(a)$. 

Otherwise let $a>0$ be the largest real number such that the triangle $(0,0),(na,a)$ and $(0,a)$ is contained in $\Omega$. Call this triangle $\Omega_1$. The line $y=a$ intersect $\partial \Omega$ in a line segment from $(x_2,a)$ to $(x_3,a)$ with $x_2\leq x_3$. Let $\Omega'_2$ denote the portion of $\Omega$ above the line $y=a$ and to the left of $x=x_2$. By applying the translation $(x,y)\rightarrow (x,y-a)$ to $\Omega'_2$ we obtain a new domain $\Omega_2$. Let $\Omega'_3$ denote the portion of $\Omega$ above the line $y=a$ and to the right of the line $x=x_3$. By first applying the translation $(x,y)\rightarrow (x-a,y-a)$ and then multiplying by $\begin{pmatrix} 0 & 1 \\ -1 & n \end{pmatrix}\in \text{SL}_2(\mathbb{Z})$ we obtain a new domain $\Omega_3$. Notice that we can interpret $X_{\Omega_2}$ and $X_{\Omega_3}$ as concave toric domains. We define the weight expansion of $\Omega$ as 
$$w(\Omega)=w(\Omega_1)\cup w(\Omega_2)\cup w(\Omega_3)$$
where in this case union means union with repetitions. 

\subsubsection{Packing Theorem}
In this section we describe and prove the ball with singularities packing. 

\begin{theo} Let $X_\Omega$ be a rational concave toric domain in $M(n,1)$ with weight expansion $w(\Omega)=(a_1,a_2\dots,a_s)$ then 
$$c_k(X_\Omega)=c_k\left(B_{n}(a_1)\amalg \coprod_{k=2}^s B(a_n)\right)$$ 
\end{theo}

\begin{proof}
Similar to \cite{Choi_2014} we can use the Traynor trick to prove that 

$$B_{n}(a_1)\amalg \coprod_{j=2}^s B(a_j)\hookrightarrow X_\Omega$$ 

Which proves that 

$$c_k\left(B_{n}(a_1))\amalg \coprod_{j=2}^s B(a_j)\right)\leq c_k(X_\Omega)$$

Notice that for $n=1$ this is just the capacities of the ball with singularities $B_n(a)$.

To prove the other inequality we can use induction over the length of the weight expansion. 

Suppose that $n>1$ and let $\Omega_1,\Omega_2$ and $\Omega_3$ be the regions as defined in the weight expansion. Let $W_1,W_2$ and $W_3$ be the disjoint union of the balls defined by the ball packings of $X_{\Omega_1}, X_{\Omega_2}$ and $X_{\Omega_3}$.

To prove the claim it is enough to prove that for every $\Lambda$ with $\mathcal{L}(\Lambda)=k$ there exist $\Lambda_1,\Lambda_2$ and $\Lambda_3$ with 
\begin{equation}
\label{sumlat}
k_1+k_2+k_3=k    
\end{equation}
where $\mathcal{L}(\Lambda_i)=k_i$ and
\begin{equation}
\label{sumnorm}
  l_{\Omega_1}(\Lambda_1)+l_{\Omega_2}(\Lambda_2)+l_{\Omega_3}(\Lambda_3)=l_{\Omega}(\Lambda)  
\end{equation}

Because it follows that
$$l_{\Omega}(\Lambda)=l_{\Omega_1}(\Lambda_1)+l_{\Omega_2}(\Lambda_2)+l_{\Omega_3}(\Lambda_3)\leq c_{k_1}(X_{\Omega_1})+c_{k_2}(X_{\Omega_2}) +c_{k_3}(X_{\Omega_3})$$
and by induction hypothesis
$$c_{k_1}(X_{\Omega_1})+c_{k_2}(X_{\Omega_2}) +c_{k_3}(X_{\Omega_3})\leq c_{k_1}(W_1)+c_{k_2}(W_2) +c_{k_3}(W_3)$$
therefore
$$l_{\Omega}(\Lambda)\leq c_{k}(W_1\amalg W_2 \amalg W_3) $$

Then the result follows from Theorem \ref{concavecapacities}.

\textit{Construction of $\Lambda_1,\Lambda_2$ and $\Lambda_3$:} 

The construction of $\Lambda_1,\Lambda_2$ and $\Lambda_3$ is similar to the $\Omega_1,\Omega_2$ and $\Omega_3$ are obtained in the definition of the weight expansion. We define $\Lambda_1$ to be the longest horizontal line contained in the compact space defined by $\Lambda$ and the lines $\{t(0,1):t\geq 0\}$ and $\{t(p,1):t\geq 0\}$. Denote by $A$ the real number such that $\Lambda_1$ hits $A(0,1)$. Notice that $\Lambda_1$ divide $\Lambda$ in two pieces $\Lambda'_1$ and $\Lambda'_2$. Define $T_2:\mathbb{R}^2\rightarrow \mathbb{R}^2$ as the translation by $(0,-a)$ and $T_3:\mathbb{R}^2\rightarrow \mathbb{R}^2$ as the map obtained when translating by $-a(1,2)$ and then multiplying by $\begin{pmatrix} 0 & 1 \\ -1 & n \end{pmatrix}$. Now we define $\Lambda_2=T_1(\Lambda_2')$ and $\Lambda_3=T_3(\Lambda_3')$.

Writte $k_i=\mathcal{L}(\Lambda_i)$. Notice that the functions $T_2$ and $T_3$ preserve lattice points. It follows that $k_1+k_2+k_3=k$. This proves (\ref{sumlat}).

To prove (\ref{sumnorm}) notice that for any vector $v$, we have that 
\begin{align*}
    l_{\Omega_2}(v)=l_{\Omega}(v)-v\times a(0,1) \\
    l_{\Omega_3}(v)=l_{\Omega}(v)-v\times a(p,1)
\end{align*}
and that $l_{\Omega_1}(0,1)=a=(1,0)\times a(0,1)=(1,0)\times a(p,1)$. So by summing over all edges of $\Lambda$ and using the equations above, we conclude that
$$l_\Omega(\Lambda)=l_\Omega(\Lambda_1)+l_\Omega(\Lambda_2)+l_\Omega(\Lambda)$$
which is the equation \eqref{sumnorm}. We concluded the proof. 
\end{proof}

Notice that by taking $n=1$ we recover \cite[Theo.~1.4]{Choi_2014}. Which is the standard ball packing theorem for concave domains. 

\subsection{Idea of the proof of Theorem \ref{concavecapacities}}

In principle, Theorem \ref{concavecapacities} can be proved using combinatorial methods, as in \cite{Choi_2014}. Here, we take a different approach. Using ideas from \cite{hutchings2006rounding,keon} we construct a combinatorial model for the embedded contact complex of a concave lens space, as explain in Proposition \ref{equicomgeo}. With this combinatorial model, the ECH capacities can be interpreted combinatorially as well. As it is explained in Section \ref{ECHspectrum} the ECH capacities of a Liouville domain are define as the ECH spectrum of its boundary. For concave toric domains in $M(n,m)$ the ECH capacities are the ECH spectrum of a concave lens space $L(n,m)$.

The embedded contact homology of a lens space allows a convenient simplification of the definition of the ECH spectrum, as explained in Corollary \ref{ECHlens} and Lemma \ref{capsimp}. Therefore, the calculation of the ECH capacities of a toric domain is equivalent to finding all the non-nullhomologous sums of generators with an even ECH index. For the case of concave lens spaces, we prove that the ECH differential can be described using a combinatorial operation called \textit{corounding the corner}, and using this combinatorial description, we can classify all the non-nullhomologous sums of generators as described at the end of Section \ref{ECH-of-Lens-Spaces}.






\section{Foundations of Embedded Contact Homology}
\label{ECHfoundations}

Let $Y$ be a  closed contact $3$-manifold with a contact form $\lambda$, that is,  $\lambda$ is a $1$-form such that $\lambda\wedge\dd \lambda>0$, and let $\xi=\ker \lambda$ be the contact structure. The Reeb vector field $R_{\lambda}$ is the unique vector field in $Y$ satisfying:
\begin{align*}
    \iota_{R_{\lambda}}\dd\lambda=0  \text{ and } & \lambda(R_\lambda)=1 
\end{align*}

We denote by $\phi_t$ the flow of $R_{\lambda}$ which is usually called the Reeb flow. A closed orbit of $\phi_t$ is called a Reeb orbit. A Reeb orbit $\gamma:\mathbb{R}/T\mathbb{Z}\rightarrow Y$ with period $T>0$ is nondegenerate when the linearized return map $P_\gamma:=\dd \phi_T|_\xi: \xi_{\gamma(0)}\rightarrow \xi_{\gamma(0)}$ does not admit $1$ as an eigenvalue. The contact form $\lambda$ is nondegenerate if all Reeb orbits are nondegenerate. Suppose that $\lambda$ is nondegenerate. Since $P_\lambda$ is a linear symplectomorphism, it turns out that the Reeb vector field admits three types of closed orbits:

\begin{enumerate}
    \item \textit{Elliptic}: orbits $\gamma$ such that the eigenalues of the linearized return map $P_\gamma$ are norm one complex numbers.
    \item \textit{Positive hyperbolic}: when the eigenvalues of $P_\gamma$ are positive real numbers.
    \item \textit{Negative hyperbolic}: When eigenvalues of $P_\gamma$ are negative real numbers. 
\end{enumerate}

An orbit set $\alpha=\{(\alpha_i,m_i)\}$ is a finite set, where $\alpha_i$ are distinct embedded Reeb orbits on $Y$ and $m_i$ are positive integers. An admissibe orbit set is an orbit set such that $m_i=1$ whenever $\alpha_i$ is hyperbolic. We denote the homology class of an orbit set $\alpha$ by 
$$[\alpha]=\displaystyle\sum m_i [\alpha_i]\in H_1(Y)$$
For a fixed $\Gamma\in H_1(Y)$, and a generic almost complex structure $J$ on $\mathbb{R}\times Y$ compatible with its symplectic structure, the chain complex $\text{ECC}_*(Y,\lambda,\Gamma,J)$ is the $\mathbb{Z}_2$-vector space generated by the admissible orbits set in homology class $\Gamma$, and its differential counts certain $J$-holomorphic curves in $\mathbb{R}\times Y$, as explained below. This chain complex gives rise to the embedded contact homology $\text{ECH}_*(Y,\lambda,
\Gamma,J)$. Taubes proved \cite{TaubesECHSW}  that $\text{ECH}_*(Y,\lambda,
\Gamma,J)$ is isomorphic to a version of Seiberg-Witten Floer cohomology $\Hat{HM}^{-*}(Y,\mathfrak{s}_\xi+PD(\Gamma))$. In particular, $\text{ECH}_*(Y,\lambda,\Gamma,J)$ does not depend on $\lambda$ or $J$, and so write $ECH_*(Y,\xi,\Gamma)$.

In the next sections we will need to consider relative homological classes that relate two different orbit set $\alpha$ and $\beta$ with the same homology. To be precise we denote by $H_2(Y,\alpha,\beta)$ the affine space over $H_2(Y)$ that consist of $2$-chains $\Sigma$ in $Y$ with 
$$\partial \Sigma=\displaystyle\sum_im_i\alpha_i-\displaystyle\sum_jn_j\beta_j$$
modulo boundaries of $3$-chains. We call $H_2(Y,\alpha,\beta)$ the \textit{relative second homology of $\alpha$ and $\beta$}.

Several of the definitions of ECH are a bit delicate. Because of that we dedicate some more  subsections to properly define the different parts that constitute 
 this homology.

\subsection{The \text{ECH} index}

We denote by $H_2(Y,\alpha,\beta)$ the affine space over $H_2(Y)$ that consist of $2$-chains $\Sigma$ in $Y$ with 
$$\partial \Sigma=\displaystyle\sum_im_i\alpha_i-\displaystyle\sum_jn_j\beta_j$$
modulo boundaries of $3$-chains. We call $H_2(Y,\alpha,\beta)$ the \textit{relative second homology of $\alpha$ and $\beta$}.

In this section we define the ECH index which is an interger number associated to a pair of Reeb sets with the same homology and a relative homological class of these two Reeb sets. The ECH index gives the gradding of the embedded contact homology. An interesting feature of the ECH index is that it is the sum of three terms that we define below, the Relative Conley-Zenhder index, the relative Chern Class and the relative intersection number, each of these terms depend on the trivialization of the contatc structure over the Reeb orbits, but the ECH index itself does not.  

Since trivialization play an essential role in the concepts we will introduce, we need to add some notations. We denote by $\mathcal{T}(\gamma)$ the set of homotopy classes of symplectic trivialization of $\xi|_{\gamma}$. This is an affine space over $\mathbb{Z}$: given two trivializations $\tau_1,\tau_2:\xi|_{\gamma}\rightarrow \mathbb{S}^1\times \mathbb{R}^2$, we denote by $\tau_1-\tau_2$ the degree of $\tau_1\circ \tau_2^{-1}:\Sp^1\rightarrow \text{Sp}(2,\mathbb{R})\cong \mathbb{S}^1$. Let $\alpha=\{(\alpha_1,m_i)\},\beta=\{(\beta_j,n_j)\}$ be two orbit sets. If $\tau\in \mathcal{T}(\alpha,\beta):=\prod_i\mathcal{T}(\alpha_i)\times \prod_j \mathcal{T}(\beta_j)$, the elements of $\mathcal{T}(\alpha_i)$ and $\mathcal{T}(\beta_j)$ are denoted by $\tau_i^+$ and $\tau_j^-$.

\subsubsection{Conley-Zenhder index}
\label{Conley-Zenhder}

Now we define the relative Conley-Zenhder index which roughly counts how much a flow turn near a Reeb orbit with respect to a certain trivialization. Let $\gamma:\mathbb{R}/T\mathbb{Z}\rightarrow Y$ be a parametrized Reeb and  $\tau$ a trivialization of $\gamma$. If $\phi_t$ is the Reeb flow, the derivative
$$\dd\phi_t:T_{\gamma(0)} Y\rightarrow T_{\gamma(t)}Y$$
restricts to a linear symplectomorphism $\psi_t:\xi_{\gamma(0)}\rightarrow \xi_{\gamma(t)}$. Using the trivialization $\tau$, the later can be viewed as a $2\times 2$ symplectic matrix for each $t$. Since $\lambda$ is nondegenerate, this give rise to a path of symplectic matrices starting at the identity $I_{2\times 2}$ and ending at the linearized return map $\psi_T=P_{\gamma}$, which does not have $1$ as an eigenvalue. So the Conley-Zehnder index $CZ_\tau(\gamma)\in \mathbb{Z}$ is defined as the Conley-Zehnder index of the path $\{\psi_t\}_{t\in [0,T]}$. In dimension four this index can be explicity defined as follows. 
If $\gamma$ is hyperbolic, let $v\in \mathbb{R}^2$ be an eigenvector of $P_\gamma$, then the family of vectors $\{\Psi_t(v)\}_{t\in [0,T]}$  rotates by angle $\pi k$ for some integer $k$ (which is even in the positive hyperbolic case and odd in the negative hyperbolic case), and
\begin{equation}
    \label{CZhyperbolic}
    \text{CZ}_{\tau}(\gamma)=k    
\end{equation}

If $\gamma$ is elliptic, then we can change the trivialization so that each $\psi_t$ is rotation by angle $2\pi \theta_t\in \mathbb{R}$ where $\theta_t$ is a continuous function of $t\in [0,T]$ and $\theta_0=0$. The number $\theta=\theta_T\in \mathbb{R}/\mathbb{Z}$ is called the `rotation angle' of $\gamma$ with respect to $\tau$, and 
\begin{equation}
    \label{CZelliptic}
    \text{CZ}_\tau(\gamma)=2 \lfloor \theta\rfloor+1
\end{equation}
If one changes the trivialization $\tau$ by another $\tau'$, the Conley-Zehnder index changes in the following way:
$$\text{CZ}_\tau(\gamma^k)-\text{CZ}_\tau(\gamma^k)=2k(\tau-\tau')$$

\subsubsection{Relative first Chern class}

Let $Z\in H_2(Y,\alpha,\beta)$ and $\tau\in \mathcal{T}(\alpha,\beta)$. Given a surface $S$ with boundary and a smooth map $f:S\rightarrow Y$ representing $Z$, the relative first Chern class $c_\tau(Z)=c_1(\xi|_{f(s)},\tau)\in \mathbb{Z}$ is defined as the signed count of zeros of a generic section $\phi$ of $f^*\xi$ that is trivial with respect to $\tau$.

The function $c_\tau$ is linear relative to the homology class, that is If $Z\in H_2(Y,\alpha,\beta)$ and $Z'\in H_2(Y,\alpha',\beta')$ then
\begin{align}
    c_\tau(Z+Z')=c_\tau(Z)+c_\tau(Z')
\end{align}

Moreover, if we change the trivialization $\tau$ by a trivialization $\tau'$, then 
\begin{align}
    c_\tau(Z)-c_{\tau'}(Z')=\displaystyle\sum_im_i({\tau'}_i^+-\tau_i^+)-\displaystyle\sum_j n_j ({\tau'}_j^--\tau_j^-)
\end{align}

\subsubsection{Relative intersection number}
    Let $\pi_Y:\mathbb{R}\times Y\rightarrow Y$ denote the projection and take a smooth map $f:S\rightarrow [-1,1]\times Y$, where $S$ is compact oriented surface with boundary, such that $f|_{\partial S}$ consists of positively oriented covers of $\{1\}\times \alpha_i$ with multiplicity $m_i$ and negatively oriented covers of $\{-1\}\times \beta_j$ with multiplicity $n_j$, $\pi_Y\circ f$ represents $Z$, the restriction $f|_{\Dot{S}}$ to the interior of $S$ is an embedding, and $f$ is transverse to $\{-1,1\}\times Y$. Such an $f$ is called an \textit{admissible representative} for $Z\in H_2(Y,\alpha,\beta)$ and we abuse notation by denoting this representative as $S$. Futhermore, suppose that $\pi_Y|_S$ is an immersion near $\partial S$ and $S$ contains $m_i$ (resp. $n_j$) singly convered circles at $\{1\}\times \alpha_i$ (resp. $\{-1\}\times \beta_j$), given by projecting conormal vectors in $S$, are $\tau$-trivial. Moreover, in each fiber $\xi$ over $\alpha_i$ or $\beta_j$, these sections lie in distinct rays. Then $S$ is a $\tau$-\textit{representative}.

Let $\tau\in \mathcal{T}(\alpha\cup\alpha',\beta\cup\beta')$ be a trivialization and $S, S'$ be $\tau$-representatives of $Z\in H_2(Y,\alpha,\beta)$ and $Z'\in H_2(Y,\alpha',\beta')$ respectively, such that the projected conormal vectors at the boundary all lie in different rays. Then $Q_\tau(Z,Z')\in \mathbb{Z}$  is the signed count of (transverse) intersections of $S$ and $S'$ in $(-1,1)\times Y$. Also $Q_\tau$ is quadratic in the following sense
\begin{equation}
    Q_\tau(Z+Z')=Q_\tau(Z)+2Q_\tau(Z,Z')+Q_\tau(Z')
\end{equation}
If $Z=Z'$ we write $Q_\tau(Z):=Q_\tau(Z,Z)$. It can be proven (see for example \cite{hutchings2014lecture}) that 
\begin{equation}
    \label{relativeQtau}
    Q_\tau(Z)=c_1(N,\tau)-w_\tau(S)
\end{equation}
where $c_1(N,\tau)$, the \textit{relative Chern number of the normal bundle}, is a signed count of zeros of a generic section of $N|_S$ such that the restriction of this section to $\partial S$ agrees with $\tau$; note that the normal bundle $N$ can be canonically identified with $\xi$ along $\partial S$. Meanwhile, the term $w_\tau(S)$, the \textit{asymptotic writhe}, is defined by using the trivialization $\tau$ to identify a neighborhood of each Reeb orbit with $\Sp^1\times D^2\subset \mathbb{R^3}$, and then computing the writhe at $s>>0$ slice of $S$ near the boundary using this identification.

Finally, if $Z,Z'\in H_2(Y,\alpha,\beta)$, changing the trivialization yields
\begin{equation}
    Q_\tau(Z,Z')-Q_\tau(Z,Z')=\displaystyle\sum_i m_i^2 ({\tau'}_i^+-\tau_i^+) -\displaystyle\sum_j n_j^2({\tau'}_j^+-\tau_j^+)
\end{equation}

With the above definitions in place we can properly define the ECH index

\begin{defi}
    Let $\alpha=\{(\alpha_i,m_i)\},\beta=\{(\beta_j,n_j)\}$ be two orbit sets in the homology class $\Gamma$ and $Z\in H_2(Y,\alpha,\beta)$. The $\text{ECH}$ index is defined by 
    \begin{equation}
        \label{ECHindex}
        I(\alpha,\beta,Z)=c_\tau(Z)+Q_\tau(Z)+\text{CZ}_\tau^I(\alpha)-CZ_\tau^I(\beta)
    \end{equation}
    where $\text{CZ}^I_\tau(\alpha)=\sum_i\sum_{k=1}^{m_i}\text{CZ}_\tau(\alpha_i^{k_i})$ and similarly for $\text{CZ}^I_\tau(\beta)$.
\end{defi}

    \begin{pro}[\cite{hutchings2014lecture} section 3.4]
        \label{IndexECHproper}
        The $\text{ECH}$ index has the following properties:
        \begin{enumerate}[a)]
            \item (\textit{Well defined}) $I(\alpha,\beta,Z)$ does not depend on $\tau$, although each term of the formula does.
            \item (\textit{Additivity}) $I(\alpha,\beta,Z+W)=I(\alpha,\delta,Z)+I(\delta,\beta,W)$, whenever $\delta$ is another orbit set in $\Gamma$, $Z\in H_2(Y,\alpha,\beta)$ and $W\in H_2(Y,\delta,\beta)$.
            \item (\textit{Index parity}) If $\alpha$ and $\beta$ are chain complex generators, then 
            $$(-1)^{I(Z)}=\epsilon(\alpha)\epsilon(\beta)$$
            where $\epsilon(\alpha)$ denotes minus to the number of positive hyperbolic orbits in $\alpha$ and similarly $\epsilon(\beta)$.
            \item (\textit{Index ambiguity Formula}) $I(\alpha,\beta,Z)-I(\alpha,\beta,Z')=\langle c_1(\xi)+2 \text{PD}(\Gamma),Z-Z' \rangle$ where $c_1(\xi)$ is the first Chern class of the vector bundle $\xi$ and $\text{PD}$ denotes the Poincare dual.
        \end{enumerate}
    \end{pro}


\subsection{Fredholm index, Differential and Grading}
While the ECH index gives the gradding of the embedded contact homology, the Fredholm index gives the dimension of the moduli space of the J-holomorphic currents that we are intereted in counting. For a generic almost complex structure $J$ the Fredholm index of a $J$-holomorphic curve $C$ is defined as (see \cite[Sec~3.2]{hutchings2014lecture} for details): 
\begin{align}
    \label{FredholmIndex}
    \text{ind}(C)=-\chi(C)+2 c_\tau(C)+\displaystyle\sum_{i=1}^k \text{CZ}_\tau(\gamma_i^+)-\displaystyle\sum_{j=1}^l\text{CZ}_\tau(\gamma_j^-)
\end{align}
where $\chi(C)$ denotes the Euler characteristic of the $J$-holomorphic curve $C$ with $k$ positive ends at the Reeb orbits $\gamma_1^+\dots\gamma_k^+$ and the $l$ negative ends at Reeb orbits $\gamma_1^-\dots\gamma_k^-$.

The proposition below relates the ECH index with the Fredholm index when the ECH index is one or two. This relationship is one of the important step towards the definition of the differential map of the embedded contact and it also allow us to consider the U-map which we define in section \ref{Umap}. Here a trivial cylinder is $\mathbb{R}\times \gamma$, where $\gamma$ is a Reeb orbit.

\begin{pro}{\cite[Prop~3.1]{hutchings2014lecture}}
\label{BasedDiff}
Suppose $J$ is generic. Let $\alpha$ and $\beta$ be orbit sets and let $\mathcal{C}\in \mathcal{M}(\alpha, \beta)$ be any $J$-holomorphic current in $\mathbb{R}\times Y$, not necessarily somewhere injective. Then
\begin{enumerate}
    \item $I(\mathcal{C})\geq 0$, with equality if and only if $\mathcal{C}$ is a union of trivial cylinders with multiplicites. 
    \item If $I(\mathcal{C})=1$ then $\mathcal{C}=\mathcal{C}_0\sqcup C_1$, where $I(\mathcal{C}_0)=0$, and has $\text{ind}(C_1)=I(C_1)=1$.
    \item If $I(\mathcal{C})=2$, and $\alpha$ and $\beta$ are chain complex generators, then $\mathcal{C}=\mathcal{C}_0\sqcup C_2$, where $I(\mathcal{C}_0)=0$, and has $\text{ind}(C_1)=I(C_1)=2$.
\end{enumerate}
\end{pro}

To prove the above Proposition the following property is used which is a particular case of \cite[Prop.~7.1]{hutchings2002index}: If $\mathcal{C}$ is a $J$-holomorphic current with no trivial cylinders and $T$ is the union of (possibly repeated) trivial cylinders, then
\begin{equation}
    \label{intcyl}
    I(C\cup T)\geq I(C)+2\#(C\cap T)
\end{equation}
From intersection positivity it also follows that $\#(C\cap T)\geq 0$, with equality if and only if $C$ and $T$ are disjoint. We will make use of this inequality as well. 

Given two chain complex generators $\alpha$ and $\beta$, the chain complex differential $\partial$ coefficient $\langle\partial \alpha, \beta\rangle\in \mathbb{Z}_2$ is a $\text{mod } 2$ count of $\text{ECH}$ index $1$ of $J$-holomorphic curves in the symplectization of $Y$ that \textit{converge as currents} to $\sum_i m_i\alpha_i$ as $s\rightarrow \infty$ and to $\sum_j n_j \beta_j$ as $s\rightarrow -\infty$ see e.g. \cite{hutchings2014lecture}.

It follows from Proposition \ref{BasedDiff} that $I$ gives rise to a relative $\mathbb{Z}_d$-grading on the chain complex $\text{ECC}_*(Y,\lambda,\Gamma,J)$, where $d$ is the divisibility of $c_1(\xi)+2\text{PD}(\Gamma)\in H^2(Y;\mathbb{Z})$ mod torsion. In order to define an (non-canonical) absolute $\mathbb{Z}_d$-grading, it is enought to fix some generatos $\beta$ with homology $\Gamma$ and set
$$I(\alpha,\beta):=[I(\alpha,\beta,Z)],$$
for an arbitraty $Z\in H_2(Y,\alpha, \beta)$. By additivity property 2. in Propostion \ref{BasedDiff} the differential decreases this absolute grading by $1$. Moreover, when $c_1(\xi)+2\text{PD}(\Gamma)$ is torsion in $H^2(Y;\mathbb{Z})$, we obtain a $\mathbb{Z}$ gradding on $\text{ECC}_*(Y,\lambda,\Gamma, J)$ as in the case of lens spaces.

\subsection{Additional structures for embedded contact homology.} 

In this subsection we define some additional important structures in the ECH setting that will be needed in the rest of the exposition. 

\subsubsection{$U$-map}
When $Y$ is connected, there is a well-defined ``U-map"
\begin{align}
    \label{Umap}
    U:\text{ECH}_*(Y,\xi,\gamma)\rightarrow\text{ECH}_{*-2}(Y,\xi,\Gamma)
\end{align}

This is induced by a chain map
\begin{align}
    U_{J,z}:(\text{ECC}_*(Y,\lambda,\Gamma),\partial_J)\rightarrow(\text{ECH}_{*-2}(Y,\xi,\Gamma),\partial_J)
\end{align}    
which counts $J$-holomorphic currents with $\text{ECH}$ index $2$ passing through a generic point $z\in \mathbb{R}\times Y$. The assumption that $Y$ is connected implies that the induced map on homology does not depend on the choice of base point, see \cite[Sec 2.5]{WeinsteinConjecture}  for details. Taubes showed in \cite{TaubesECHSW} Theorem 1.1 that the $U$ map induced in homology  agrees with a corresponding map on Seiberg-Witten Floer homology. We thus obtain the well-defined $U$-map (\ref{Umap}).

\begin{defi}
        A $U$-\textit{sequence} for $\Gamma$ is a sequence $\{\sigma_k\}_{k\geq 1}$where each $\sigma_k$ is a nonzero homogenous class in $\text{ECH}_*(Y,\xi,\Gamma)$, and $U\sigma_{k+1}=\sigma_k$ for each $k\geq 1$.
\end{defi}
We will need the following nontriviality result for the $U$-map, which is proved by combining Taubes' isomorphism with a result from Kromheimer-Mrowka \cite{kronheimer2007monopoles}:

\begin{pro}{\cite[Prop. 2.3]{cristofaro2019torsion}}
    If $c_1(\xi)+2\text{PD}(\Gamma)\in H^2(Y,\mathbb{Z})$ is torsion, then a $U$-sequence for $\Gamma$ exists. 
\end{pro}
\subsubsection{The $\text{ECH}$ partition conditions}

The $\text{ECH}$ partition conditions are a topological type data associated to the pseudoholomorphic curves (and currents) which can be obtained indirectly from certain $\text{ECH}$ index relations. In particular, the covering multiplicities of the Reeb orbits at the ends of the non-trivial components of the pseudoholomorphic curves (and currents) are uniquely determined by the trivial cylinder component information. The genus can be determined by the current's relative homology class.
\begin{defi}{\cite{hutchings2014lecture}} Let $\gamma$ be an embedded Reeb orbit and $m$ a positive integer. We define two partitions of $m$, the \textit{positive partition} $P_\gamma^+(m)$ and the \textit{negative partition} $P_\gamma^-(m)$ as follows
\begin{itemize}
    \item If $\gamma$ is positive hyperbolic, then
    $$P_\gamma^+(m):=P_{\gamma}^-(m):=(1,\dots,1)$$
    \item If $\gamma$ is negative hyperbolic, then
    \begin{align*}
        P_\gamma^+(m):=P_{\gamma}^-(m):=\begin{cases}
            (2,\dots,2) & m \text{ even}\\
            (2,\dots,2,1) & m \text{ odd}
        \end{cases}
    \end{align*}
    \item If $\gamma$ is elliptic then the partitions are defined in terms of the quantity $\theta\in \mathbb{R}/\mathbb{Z}$ for which $\text{CZ}_\tau(\gamma^k)=2\lfloor k\theta \rfloor+1$. We write
    $$P_\gamma^\pm(m):=P_\theta^\pm(m)$$
    with the right hand side defined as follows. 
\end{itemize}
    Let $\Gamma_\theta^+(m)$ denote the highest concave polygonal path in the plane that starts at $(0,0)$, ends at $(m,\lfloor k\theta \rfloor)$, stays below the line $y=\theta x$ and has corners at lattice points. Then the integers $P_\theta^+(m)$ are the horizontal displacements of the segments of the path $\Gamma_\theta^+(m)$ between the lattice points. 
    
    Likewise, let $\Gamma_\theta^+(m)$ denote the lowest convex polygonal path in the plane that starts at $(0,0)$, ends at $(m,\lfloor k\theta \rfloor)$, stays above the line $y=\theta x$ and has corners at lattice points. Then the integers $P_\theta^-(m)$ are the horizontal displacements of the segments of the path $\Gamma_\theta^-(m)$ between the lattice points. 
    
    Both $P_\theta^\pm(m)$ depend only on the class of $\theta$ in $\mathbb{R}/\mathbb{Z}$. Moreover, $P_\theta^+(m)=P_{-\theta}^{-}(m)$.
\end{defi}

\subsubsection{Filtered ECH.}
There is a filtration on \text{ECH} which enables us to compute the embedded contact homology via succesive approximations (see theorem 2.17 \cite{nelson2022embedded} ). The \textit{symplectic action} or \textit{lenght} of an Reeb current $\alpha=\{(\alpha_i,m_i)\}$ is 
\begin{align*}
    \mathcal{A}(\alpha):=\displaystyle\sum_i m_i\int_{\alpha_i} \lambda
\end{align*}
If $J$ is a $\lambda$-compatible and there is a $J$-holomorphic current from $\alpha$ to $\beta$, then $\mathcal{A}(\alpha)\geq \mathcal{A}(\beta)$ by Stokes' theorem, since $\dd \lambda$ is an area form on such $J$-holomorphic curves. Since $\partial$ counts $J$-holomorphic currents, it decreases symplectic action, that is, 
\begin{align}
\label{Stokes}
    \langle\partial\alpha,\beta\rangle\not=0 \text{ implies } \mathcal{A}(\alpha)\geq \mathcal{A}(\beta)
\end{align}
Let $\text{ECC}_*^L(Y,\lambda,\gamma,J)$ denote the subgroup of $\text(ECC)_*(Y,\lambda,\Gamma,J)$ generated by Reeb currents of symplectic action less than $L$. Because $\partial$ decreases action, it is a subcomplex. It is shown (See \cite[theo~1.3]{hutchings2013proof}) that the homology of $\text{ECC}_*(Y,\lambda,\Gamma,J)$ is independent of $J$, therefore we denote its homology by 
$\text{ECC}^L_*(Y,\lambda,\Gamma,J)$, which we call filtered $\text{ECH}$.
Given $L<L'$, there is a homomorphism 
\begin{align*}
\iota^{L,L'}:\text{ECH}_*^L(Y,\lambda,\Gamma)\rightarrow\text{ECH}_*^{L'}(Y,\lambda,\Gamma)   
\end{align*}
induced by the inclusion $\text{ECC}_*^L(Y,\lambda,\Gamma)\rightarrow\text{ECC}_*^{L'}$ and independent of $J$. The $\iota^{L,L'}$ fit together into a direct system $(\{\text{ECC}_*^L(Y,\lambda,\Gamma)\}_{L\in \mathbb{R}},\iota^{L,L'})$. Because taking direct limits commutes with taking homology, we have
\begin{align}
\label{directlimit}
\text{ECH}_*(Y,\lambda,\Gamma)=H_*\left(\lim_{L\rightarrow \infty} \text{ECC}_*^L(Y,\lambda,\Gamma,J)\right)=\lim_{L\rightarrow\infty}\text{ECH}^L_*(Y,\lambda,\Gamma)   
\end{align}

\subsubsection{ECH spectrum}
\label{ECHspectrum}
\text{ECH} contains a canonical class defined as follows. Observe that for any nondegenerate contact three-manifold $(Y,\lambda)$, the empty set of Reeb orbits is a generator of the chain complex $\text{ECC}(Y,\lambda,0,J)$. It follows from (\ref{Stokes}) that this chain complex generator is actually a cycle, i.e,
$$\partial\emptyset=0$$

\text{ECH} cobordism maps can be used to show that the homology class of this cycle does not depend on $J$ or $\lambda$, and thus represents a well-defined class 
$$[\emptyset]\in\text{ECH}_*(Y,\xi,0)$$

\begin{defi}
    Let $(Y,\lambda)$ be a closed contact closed $3$-manifold such that $[\emptyset]\not=0\in\text{ECH}(Y,\xi,0)$. We define the \textit{ECH spectrum} as
\begin{align*}
        c_k(Y,\lambda)=\inf\{L:\eta\in \text{ECH}^L_{2k}(Y,\lambda,0), U^k \eta=[\emptyset]\}
    \end{align*}
\end{defi}

In the most important cases for our purpouses we can use a simpler version of the $\text{ECH}$ spectrum. We say that a closed sum of generators $\alpha_1+\cdots+\alpha_r$ is \textit{minimal} if after removing any number of summands the sum is no longer closed. 

\begin{lemm}
\label{capsimp}
Let $(Y,\lambda)$ be a contact closed 3-manifold. Suppose that 
\begin{align*}
\text{ECH}_*(Y,\lambda,0) = \left\{ 
    \begin{array}{lcc}
                \mathbb{Z}_2     &   \text{if}  & *=2k \\
                       0         &   \text{if}  & *=2k+1
             \end{array}
   \right.
\end{align*}
and that the $U$ map of $\text{ECH}(Y,\lambda)$ is an isomorphism for every even index. Then
\begin{equation}
    \label{specsimp}
    \begin{split}
    c_k(Y,\lambda)=\min\{\max\{\mathcal{A}(\alpha_1),\cdots,\mathcal{A}(\alpha_r)\}: 
    I(\alpha_1)=\cdots=I(\alpha_r)=2k,  \\
    \alpha_1+\cdots+\alpha_r \text{ is minimal and non-nullhomologous}\}
    \end{split}
\end{equation}
\end{lemm}

\begin{proof} Since $U$ is an isomorphism and $ECH_{2k}=\mathbb{Z}_2$, the expression for the capacities simplifies as 
\begin{align*}
        c_k(Y,\lambda)=\inf\{L:\eta\in \text{ECH}_{2k}^L(Y,\lambda), \eta\not=0\}
    \end{align*}
Take $L>0$, suppose that $\eta\in \text{ECH}_{2k}^L(Y,\lambda)$ and $\eta\not=0$.

Suppose that $\eta=[\alpha_1+\cdots +\alpha_r]$ then by definition $\max\{\mathcal{A}(\alpha_1),\dots,\mathcal{A}(\alpha_r)\}<L$. It follows that 
$$\min\{\max\{\mathcal{A}(\alpha_1),\dots,\mathcal{A}(\alpha_r)\}:\eta=[\alpha_1+\cdots +\alpha_r]\}\leq c_k(Y,\lambda)$$

Take $L$ equal to the left side of the above inequality then for every $\epsilon>0$ there exist a sum of generators $\alpha_1+\cdots+\alpha_r$ such that $[\alpha_1+\cdots+\alpha_r]=\eta\in \text{ECH}^{L+\epsilon}(Y,
\lambda)$ then $c_k(Y,\lambda)\leq L$. The result follows. 
\end{proof}

By Corollary \ref{ECHlenscol} and Equation (\ref{specsimp}) holds for any lens space. Finally, we can give the definition of the $\text{ECH}$ capacities. 

\begin{defi}
    A (four-dimensional) \textit{Liouville domain} is a weakly exact symplectic filling $(X,\omega)$ of a contact three-manifold $(Y,\lambda)$.
\end{defi}

\begin{defi} If $(X,\omega)$ is a four-dimensional Liouville domain with boundary $(Y,\lambda)$, define the $\text{ECH}$ $\textit{capacities}$ of $(X,\omega)$ by
$$c_k(X,\omega)=c_k(Y,\lambda)\in [0,\infty]$$
\end{defi}

A justification for this definition can be found in \cite{hutchings2014lecture} in Section 1.5. The definition of capacities can be extended to non-Liouville domains through a limiting argument. 



\section{Embedded Contact Complex of Concave Lens Spaces}
\label{ECHlensspace}

In this section we compute the embedded contact complex of a concave contact form $\lambda$ over a lens space $L(n,m)$ with toric symmetry. We do so by given a combinatorial model to the embedded contact complex after perturbing the contact form $\lambda_a$ twice in such a way that under a bounded action every orbit is non-degenerate and we can use the direct limit property \eqref{directlimit}. As it is usual in this context (see \cite[Sec~4.2]{hutchings2014lecture}) Reeb orbits appear in $\Sp^1$-families and they can be transformed under a small perturbation into a couple of orbits. In our situation we do an addional perturbation that we are calling a \textit{concave perturbation} which is going to be useful in simplifying the combinatorial complex. 

\subsection{Reeb Dynamics for toric contact closed $3$-manifolds.}
Let $a=(a_1,a_2):[0,1]\rightarrow \mathbb{R}^2$ be a function such that it induce a concave contact form $\lambda_a$ over the lens space $L(n,m)$ as explained in section \ref{NotableSubs}. Notice that it also induce a region $\Omega$ in $V_{n,m}$. Similar to \cite{Choi_2014} Section 3.3 the closed orbits of the Reeb field associated to $\lambda_a$ are given by the following:

\begin{itemize}
    \item The circle $e_+$ obtained by the projection $\pi(\{0\}\times \To^2)$ over $L(n,m)$ with action $\mathcal{A}(e_+)=a(0)\times (n,m)$.
    \item The circle $e_-$ obtained by the projection $\pi(\{1\}\times \To^2)$ over $L(n,m)$ with action $\mathcal{A}(e_-)=a_2(1)$.
    \item For each $x\in (0,1)$ for which $(a_1'(x),a'_2(x))$ is proportional to $(-p,q)$ where $p$ and $q$ are relative primes to each other, there is a Morse-Bott $\mathbb{S}^1$-family of Reeb orbits foliating $\{x\}\times \To^2$, with relative homology over $\To^2$ equal $(-p,q)$. Each orbit of this folliation has action $a(x)\times (p,q)$.
\end{itemize}

Notice that the actions of the Reeb orbits just decribed coincide with the $\Omega$-length given in the Definition \ref{omegalength}.


\begin{rem}
\label{notationorbits}
Depending on the convinience we can use different notations to denote the Reeb Orbits. The notations $e_{p,q}, h_{p,q}$ means the elliptic or hyperbolic Reeb orbit with homology $(p,q)$ respectively. We can also write $e_x, h_x$ to mean the elliptic or hyperbolic Reeb orbit ocurring at $x$ where $a'(x)$ is proportional to a primitive vector. We write $e^m_x$ to mean the ellitic orbit at $x\in [0,1]$ with multiplicity $m$ and $h^m_x$ to mean the orbit set $\{(h_x,1),(e_x,m-1)\}$. So a Reeb orbit set $\alpha=\{(\alpha_i,m_i)\}$ can be written with multiplicative notation in a unique way as $\alpha=f_{x_1}^{m_1}\cdots f_{x_k}^{m_k}$ with $x_1< \cdots< x_k$ where each $f_i$ is a label `$e$' or `$h$'.
\end{rem}

\subsection{Two Steps Perturbation.}
\label{perturbation}

Suppose that $Y=L(n,m)$ is a lens space with a concave contact form $\lambda_a$. To obtain a simple version of a combinatorial complex it is convinient to do two perturbations over $\lambda_a$. Proposition \ref{equicomgeo} explains the exact relationship obtained under the perturbations and the embedded contact complex. 

Before explaining the perturbations it is convenient to notice that to choice a homology class of the special orbits $e_+$ and $e_-$ in $H_1(\To^2)$ is equivalent to choose a trivialization of the contact structure $\xi|_{e_+}$ and $\xi|_{e_-}$ respectively.

\textbf{Trivialization over the elliptic orbits $e_+$ and $e_-$:}

Choose a vector $(v_1,v_2)\in \mathbb{Z}^2$ such that $(p,q)\times (v_1,v_2)=1$ then $(v_1,v_2)$ induce a trivialization over the $e_+$. Similarly the vector $(-1,0)$ induces a trivialization over the $e_-$ orbit.  
\begin{lemm}
\label{rotnumbers}
The orbits $e_+$ and $e_-$ are elliptic orbits. Futhermore, with the trivialization induced $(v_1,v_2)$ and $(-1,0)$ explained above the rotation numbers $\phi_+$ and $\phi_-$ of $e_+$ and $e_-$ respectively are given by the equations
\begin{align}
    \phi_+=\displaystyle\frac{a'(0)\times (v_1,v_2)}{a'(0)\times (p,q)} && \phi_-=\displaystyle\frac{a_2'(1)}{a_1'(1)}
\end{align}
\end{lemm}
\begin{proof}
    For the orbit $e_2$ the result follows by noticing that a neighborhood of this orbits is strictly contactomorphic to a neighborhood of the sphere $\Sp^3$ with the appropiate contact structure. 
    
    To prove the corresponding claim to the orbit $e_+$  we can use the matrix 
    
    \begin{align}
    \label{trivmatrix}
    A_{(v_1,v_2)}=\begin{pmatrix}v_2 & -v_1 \\-q   & p  \end{pmatrix}    
    \end{align}
    
    in $\text{SL}_2(\mathbb{Z})$. The matrix $A$ induces a contactomophism between a neighborhood of the orbit $e_1$ to a neighborhood to the corresponding orbit in $\Sp^3$ with a contact structure induced by $A$ and the curve $a$. Futhermore, the matrix sends $(p,q)$ to $(1,0)$ and $(v_1,v_2)$ to $(0,1)$ meaning that the matrix $A$ sends the trivialization defined in $Y_a$ to the standard trivialization in $\Sp^3$. Since 
 $$A(a'_1(0),a'_2(0))=(a'(0)\times (v_1,v_2), a'(0)\times (p,q))$$
 the lemma follows. 
\end{proof}

\subsubsection{Perturbation over the concavity.}
\label{concavepert}
    Choose $(v_1,v_2)$ as in Lemma \ref{rotnumbers} to trivilize the contact structure over $e_+$. From the function $a:[0,1]\rightarrow \mathbb{R}^2$ we can have a family of smooth functions $\{a^\epsilon\}_{\epsilon>0}$ such that 
 \begin{itemize}
    \item $a^\epsilon(x)\times (a^\epsilon)'(x)<0$ for every $x\in [0,1]$
     \item $a(x)=a^\epsilon(x)$ for every $x\in [\epsilon,1-\epsilon]$.
     \item There exist positive constants $k_1$ and $k_2$ such that $a'_\epsilon(0)\rightarrow -k_1 (n,m)$ and $a'_\epsilon(1)\rightarrow k_2 (0,1)$ when $\epsilon\rightarrow 0$.
 \end{itemize}
    Write $\lambda_{a}^{\epsilon}:=\lambda_{a_\epsilon}$. Notice that the  family of contact forms $\{\lambda_a^\epsilon\}_{\epsilon>0}$ possessess the following property: for each $k_0>0$ there exists an $\epsilon>0$ such that the orbtis $e_+$ and $e_-$ in the contact form  is such that $I(e_+)>k_0$ and $I(e_-)>k_0$. We call this family of perturbations over $\{\lambda_a^\epsilon\}_{\epsilon>0}$ a \textit{concave perturbation} of $\lambda_a$.

\subsubsection{Morse-Bott Perturbation.}
\label{Morse-Bott}
As it is usual in this context (see \cite{hutchings2014lecture} Section 4.2) for each $L>0$ the contact form $\lambda_a$ can be perturbed in such a way that each $\Sp^1$-family of Reeb orbits described above with action strictly less than $L$ becomes two embedded Reeb orbits of approximately the same action. More precisely, suppose that $\{x\}\times \To^2$ is foliated by a $\Sp^1$-family of Reeb orbits with homology $(p,q)$ given by the contact form of $\lambda_a$ and action less than $L$. After the perturbation the $\Sp^1$-family becomes two embedded Reeb orbits of approximately the same action, one of them is elliptic and we denote it by $e_{(p,q)}$, the other one is hyperbolic and  we denoted it by $h_{(p,q)}$. This perturbation can be chosen in such a way that the linearization is conjugated to a small negative rotation, it follows from equations \eqref{CZhyperbolic} and \eqref{CZelliptic} that $\text{CZ}_{\tau}(e_{(p,q)})=-1$ and $\text{CZ}_{\tau}(h_{(p,q)})=0$.


We will need to perfom a concave perturbation and after that we have to perform a Morse-Bott perturbation. We will denote that perturbation as $\lambda_a^{\epsilon,L}$.



\subsection{Generators}
\label{generators}

Suppose that $\alpha=\{(\alpha_i,m_i)\}$ is a set of generators which does not contain the orbits $e_+$ or $e_-$ with any multiplicity. For each $\alpha_i$ in the orbit set $\alpha$ write $[\alpha_i]=(p_i,q_i)\in H_2(\To^2)$. Notice that we can organize the orbit set $\alpha$ as $\{(\alpha_1,m_1),\dots,(\alpha_k,m_k)\}$ where $q_1/p_1<\cdots <q_k/p_k$. 

Suppose that $[\alpha]=0\in H_2(\To^2)$. Notice that this homological condition implies that there exists a unique path $(n,m)$-concave path $P_\alpha$ (see Definition \ref{(n,m)-concave path}) such that the edges are consecutive concatenations of the vectors $m_1(p_1,q_1)\dots m_k(p_k,q_k)$ and $P_\alpha$ begins at the $(n,m)$-axis and ends at the $y$-axis. Futhermore, we make 
the $(n,m)$-concave path $P_\alpha$ into a \textit{decorated $(n,m)$-concave path $P_\alpha$} by adding the letter `$h$' or the letter `$e$' to the edge $m_i(p_i,q_i)$ depending of the pair $(\alpha_i, m_i)$ has an hyperbolic orbit or not. 
\begin{figure}
    \centering
    \input{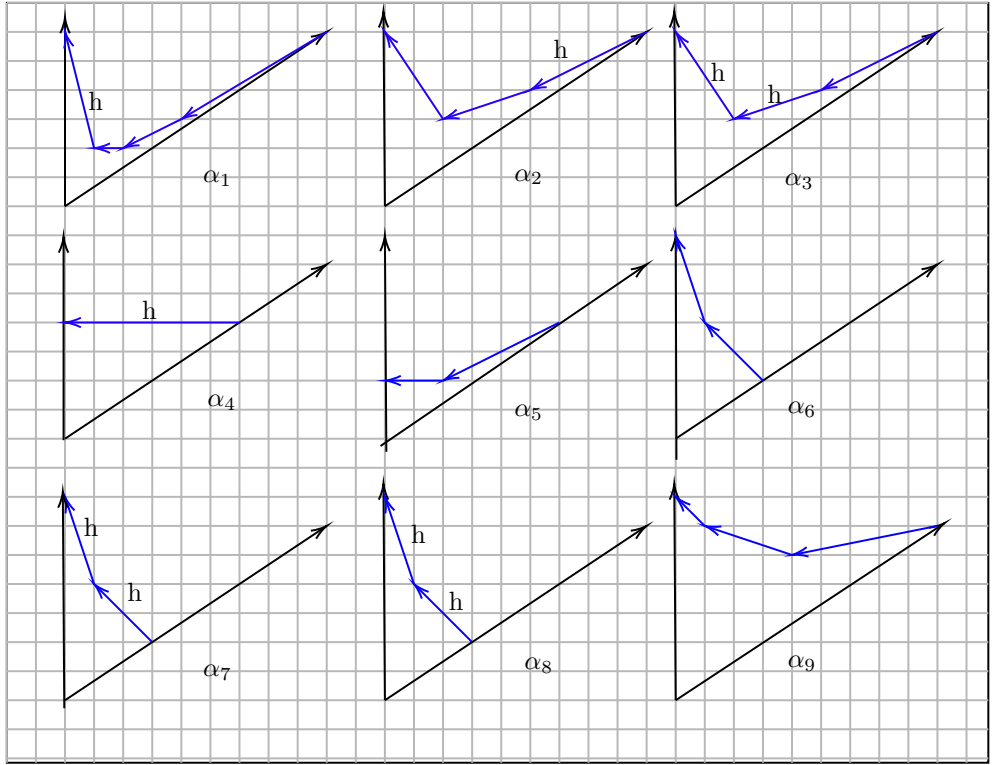}
    \caption{Examples of decorated $(n,m)$-polygonal paths.}\label{Decorated (n,m)-polygons}
\end{figure}

In Figure \ref{Decorated (n,m)-polygons} we show some examples of decorated $(n,m)$-polygons. From left to right and top to bottom these decorated $(n,m)$-polygons correspond to the orbit sets $\alpha_1=e_{(-5,-3)}\,e_{(-2,1)}\,e_{(-1,0)}\,h_{(-1,4)}$, $\alpha_2=e^2_{(-2,-1)}\,e_{(-3,-1)}\,e_{(-2,3)}$, $\alpha_3=e^2_{(-2,-1)}\,h_{(-3,-1)}\,h_{(-2,3)}$, $\alpha_4=h^6_{(-1,0)}$, $\alpha_5=h^2_{(-2,-1)}e^2_{(-1,0)}$, $\alpha_6=e^2_{(-1,1)}e_{(-1,3)}$, $\alpha_7=h^2_{(-1,1)}e_{(-1,3)}$, $\alpha_8=h^2_{(-1,1)}h_{(-1,3)}$, $e_{(-5,-1)}e_{(-3,1)}e_{(-1,1)}$. Here we are using the multiplicative notation, see Remark \ref{notationorbits}.

\subsection{Combinatorial Index.}

Consider a Reeb orbit set $\alpha$ that does not contain the orbits $e_+$ and $e_-$. Let $P_\alpha$ be its $(n,m)$-concave path. Remember from the Definition \ref{Counting points} that we denote by $\mathcal{L}_{n,m}(P_\alpha)$ the number of lattice points contained in the region defined by $P_\alpha$, the $y$-axis and the $(n,m)$-axis without counting the lattice points in the path $P_\alpha$.

We define the index of the path $P_\alpha$ as 
\begin{equation} 
\label{echinccon}
I(\alpha)=2\mathcal{L}_{n,m}(P_\alpha)+h   
\end{equation}
Where $h$ in the above equation denotes the number of `$h$' labels in $P_\alpha$. See figure \ref{Generators L(3,2)} for examples of $(3,2)$-concave paths.

\begin{figure}
    \centering
    \tikzset{every picture/.style={line width=0.75pt}} 

\begin{tikzpicture}[x=0.55pt,y=0.55pt,yscale=-1,xscale=1]

\draw  [draw opacity=0] (2,-16) -- (624,-16) -- (624,166) -- (2,166) -- cycle ; \draw  [color={rgb, 255:red, 160; green, 160; blue, 160 }  ,draw opacity=1 ] (2,-16) -- (2,166)(22,-16) -- (22,166)(42,-16) -- (42,166)(62,-16) -- (62,166)(82,-16) -- (82,166)(102,-16) -- (102,166)(122,-16) -- (122,166)(142,-16) -- (142,166)(162,-16) -- (162,166)(182,-16) -- (182,166)(202,-16) -- (202,166)(222,-16) -- (222,166)(242,-16) -- (242,166)(262,-16) -- (262,166)(282,-16) -- (282,166)(302,-16) -- (302,166)(322,-16) -- (322,166)(342,-16) -- (342,166)(362,-16) -- (362,166)(382,-16) -- (382,166)(402,-16) -- (402,166)(422,-16) -- (422,166)(442,-16) -- (442,166)(462,-16) -- (462,166)(482,-16) -- (482,166)(502,-16) -- (502,166)(522,-16) -- (522,166)(542,-16) -- (542,166)(562,-16) -- (562,166)(582,-16) -- (582,166)(602,-16) -- (602,166)(622,-16) -- (622,166) ; \draw  [color={rgb, 255:red, 160; green, 160; blue, 160 }  ,draw opacity=1 ] (2,-16) -- (624,-16)(2,4) -- (624,4)(2,24) -- (624,24)(2,44) -- (624,44)(2,64) -- (624,64)(2,84) -- (624,84)(2,104) -- (624,104)(2,124) -- (624,124)(2,144) -- (624,144)(2,164) -- (624,164) ; \draw  [color={rgb, 255:red, 160; green, 160; blue, 160 }  ,draw opacity=1 ]  ;
\draw [color={rgb, 255:red, 0; green, 0; blue, 0 }  ,draw opacity=1 ]   (22,124) -- (200.34,5.11) ;
\draw [shift={(202,4)}, rotate = 146.31] [color={rgb, 255:red, 0; green, 0; blue, 0 }  ,draw opacity=1 ][line width=0.75]    (10.93,-3.29) .. controls (6.95,-1.4) and (3.31,-0.3) .. (0,0) .. controls (3.31,0.3) and (6.95,1.4) .. (10.93,3.29)   ;
\draw [color={rgb, 255:red, 0; green, 0; blue, 0 }  ,draw opacity=1 ]   (22,124) -- (22,6) ;
\draw [shift={(22,4)}, rotate = 90] [color={rgb, 255:red, 0; green, 0; blue, 0 }  ,draw opacity=1 ][line width=0.75]    (10.93,-3.29) .. controls (6.95,-1.4) and (3.31,-0.3) .. (0,0) .. controls (3.31,0.3) and (6.95,1.4) .. (10.93,3.29)   ;
\draw [color={rgb, 255:red, 31; green, 0; blue, 255 }  ,draw opacity=1 ]   (82,84) -- (44,84) ;
\draw [shift={(42,84)}, rotate = 360] [color={rgb, 255:red, 31; green, 0; blue, 255 }  ,draw opacity=1 ][line width=0.75]    (10.93,-3.29) .. controls (6.95,-1.4) and (3.31,-0.3) .. (0,0) .. controls (3.31,0.3) and (6.95,1.4) .. (10.93,3.29)   ;
\draw [color={rgb, 255:red, 31; green, 0; blue, 255 }  ,draw opacity=1 ]   (42,84) -- (22.63,25.9) ;
\draw [shift={(22,24)}, rotate = 71.57] [color={rgb, 255:red, 31; green, 0; blue, 255 }  ,draw opacity=1 ][line width=0.75]    (10.93,-3.29) .. controls (6.95,-1.4) and (3.31,-0.3) .. (0,0) .. controls (3.31,0.3) and (6.95,1.4) .. (10.93,3.29)   ;
\draw [color={rgb, 255:red, 0; green, 0; blue, 0 }  ,draw opacity=1 ]   (242,124) -- (242,6) ;
\draw [shift={(242,4)}, rotate = 90] [color={rgb, 255:red, 0; green, 0; blue, 0 }  ,draw opacity=1 ][line width=0.75]    (10.93,-3.29) .. controls (6.95,-1.4) and (3.31,-0.3) .. (0,0) .. controls (3.31,0.3) and (6.95,1.4) .. (10.93,3.29)   ;
\draw [color={rgb, 255:red, 0; green, 0; blue, 0 }  ,draw opacity=1 ]   (242,124) -- (420.34,5.11) ;
\draw [shift={(422,4)}, rotate = 146.31] [color={rgb, 255:red, 0; green, 0; blue, 0 }  ,draw opacity=1 ][line width=0.75]    (10.93,-3.29) .. controls (6.95,-1.4) and (3.31,-0.3) .. (0,0) .. controls (3.31,0.3) and (6.95,1.4) .. (10.93,3.29)   ;
\draw [color={rgb, 255:red, 22; green, 0; blue, 255 }  ,draw opacity=1 ]   (302,84) -- (243.66,45.11) ;
\draw [shift={(242,44)}, rotate = 33.69] [color={rgb, 255:red, 22; green, 0; blue, 255 }  ,draw opacity=1 ][line width=0.75]    (10.93,-3.29) .. controls (6.95,-1.4) and (3.31,-0.3) .. (0,0) .. controls (3.31,0.3) and (6.95,1.4) .. (10.93,3.29)   ;
\draw [color={rgb, 255:red, 0; green, 0; blue, 0 }  ,draw opacity=1 ]   (442,124) -- (442,6) ;
\draw [shift={(442,4)}, rotate = 90] [color={rgb, 255:red, 0; green, 0; blue, 0 }  ,draw opacity=1 ][line width=0.75]    (10.93,-3.29) .. controls (6.95,-1.4) and (3.31,-0.3) .. (0,0) .. controls (3.31,0.3) and (6.95,1.4) .. (10.93,3.29)   ;
\draw [color={rgb, 255:red, 0; green, 0; blue, 0 }  ,draw opacity=1 ]   (442,124) -- (620.34,5.11) ;
\draw [shift={(622,4)}, rotate = 146.31] [color={rgb, 255:red, 0; green, 0; blue, 0 }  ,draw opacity=1 ][line width=0.75]    (10.93,-3.29) .. controls (6.95,-1.4) and (3.31,-0.3) .. (0,0) .. controls (3.31,0.3) and (6.95,1.4) .. (10.93,3.29)   ;
\draw [color={rgb, 255:red, 65; green, 0; blue, 253 }  ,draw opacity=1 ]   (502,84) -- (443.66,45.11) ;
\draw [shift={(442,44)}, rotate = 33.69] [color={rgb, 255:red, 65; green, 0; blue, 253 }  ,draw opacity=1 ][line width=0.75]    (10.93,-3.29) .. controls (6.95,-1.4) and (3.31,-0.3) .. (0,0) .. controls (3.31,0.3) and (6.95,1.4) .. (10.93,3.29)   ;

\draw (17,134.4) node [anchor=north west][inner sep=0.75pt]    {$I( \Lambda ) =\ 10$};
\draw (231,137.4) node [anchor=north west][inner sep=0.75pt]    {$I( \Lambda ) =16$};
\draw (470,45) node [anchor=north west][inner sep=0.75pt]   [align=left] {h};
\draw (431,139.4) node [anchor=north west][inner sep=0.75pt]    {$I( \Lambda ) =17$};

\end{tikzpicture}
    \caption{Some examples of generators for $L(3,2)$}
    \label{Generators L(3,2)}
\end{figure}
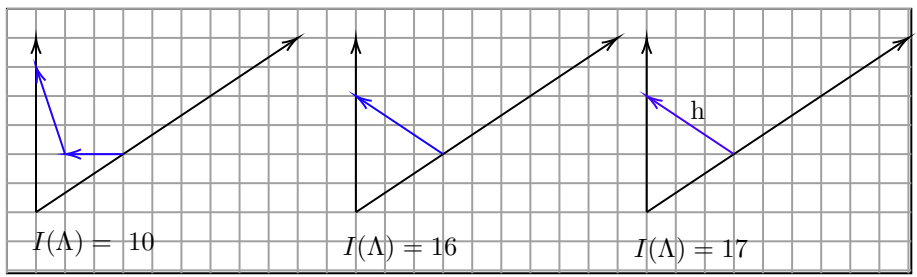

\subsection{Corounding the corner and the differential.}
\label{corounding}

In this section we explain the operation of corounding the corner. Versions of this operation are common in certain settings, see for example \cite{cristofaro2020proof, hutchings2005periodic, hutchings2006rounding}. 

With this operation over the polygonal paths we can define the combinatorial differential that correspond to the $\text{ECH}$ differential. Suppose that $\alpha$ and $\beta$ are admissible orbits set and let $P_\alpha$ and $P_\beta$ be the corresponding polygonal paths. Notice that $P_\alpha$ and the axes define a non-compact convex region $R_\alpha$, see Figure \ref{NonCompactCompactRegion}. Similarly, we have a region $R_\beta$ associated to $P_\beta$. We say that $P_\alpha$ is obtained from $P_\beta$ by \textbf{rounding a corner}, if $R_\alpha$ is a region obtained from the region $R_\beta$ after removing a corner of $P_\beta$. 

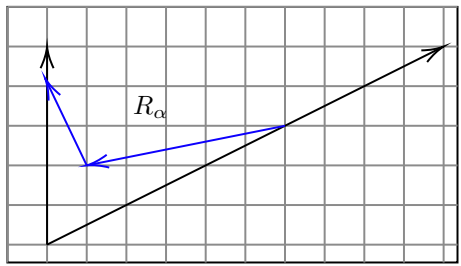
\begin{figure}
    \centering
    \tikzset{every picture/.style={line width=0.75pt}} 

\begin{tikzpicture}[x=0.75pt,y=0.75pt,yscale=-1,xscale=1]

\draw  [draw opacity=0] (242,5) -- (469,5) -- (469,134) -- (242,134) -- cycle ; \draw  [color={rgb, 255:red, 139; green, 139; blue, 139 }  ,draw opacity=1 ] (242,5) -- (242,134)(262,5) -- (262,134)(282,5) -- (282,134)(302,5) -- (302,134)(322,5) -- (322,134)(342,5) -- (342,134)(362,5) -- (362,134)(382,5) -- (382,134)(402,5) -- (402,134)(422,5) -- (422,134)(442,5) -- (442,134)(462,5) -- (462,134) ; \draw  [color={rgb, 255:red, 139; green, 139; blue, 139 }  ,draw opacity=1 ] (242,5) -- (469,5)(242,25) -- (469,25)(242,45) -- (469,45)(242,65) -- (469,65)(242,85) -- (469,85)(242,105) -- (469,105)(242,125) -- (469,125) ; \draw  [color={rgb, 255:red, 139; green, 139; blue, 139 }  ,draw opacity=1 ]  ;
\draw    (262,125) -- (262,89) -- (262,27) ;
\draw [shift={(262,25)}, rotate = 90] [color={rgb, 255:red, 0; green, 0; blue, 0 }  ][line width=0.75]    (10.93,-3.29) .. controls (6.95,-1.4) and (3.31,-0.3) .. (0,0) .. controls (3.31,0.3) and (6.95,1.4) .. (10.93,3.29)   ;
\draw    (262,125) -- (460.21,25.89) ;
\draw [shift={(462,25)}, rotate = 153.43] [color={rgb, 255:red, 0; green, 0; blue, 0 }  ][line width=0.75]    (10.93,-3.29) .. controls (6.95,-1.4) and (3.31,-0.3) .. (0,0) .. controls (3.31,0.3) and (6.95,1.4) .. (10.93,3.29)   ;
\draw [color={rgb, 255:red, 13; green, 0; blue, 253 }  ,draw opacity=1 ]   (282,85) -- (261.86,42.8) ;
\draw [shift={(261,41)}, rotate = 64.49] [color={rgb, 255:red, 13; green, 0; blue, 253 }  ,draw opacity=1 ][line width=0.75]    (10.93,-3.29) .. controls (6.95,-1.4) and (3.31,-0.3) .. (0,0) .. controls (3.31,0.3) and (6.95,1.4) .. (10.93,3.29)   ;
\draw [color={rgb, 255:red, 13; green, 0; blue, 253 }  ,draw opacity=1 ]   (382,65) -- (283.96,84.61) ;
\draw [shift={(282,85)}, rotate = 348.69] [color={rgb, 255:red, 13; green, 0; blue, 253 }  ,draw opacity=1 ][line width=0.75]    (10.93,-3.29) .. controls (6.95,-1.4) and (3.31,-0.3) .. (0,0) .. controls (3.31,0.3) and (6.95,1.4) .. (10.93,3.29)   ;

\draw (304,48.4) node [anchor=north west][inner sep=0.75pt]    {$R_{\alpha }$};

\end{tikzpicture}
    \caption{A convex region $R_\alpha$ in $L(2,1)$.}
    \label{NonCompactCompactRegion}
\end{figure}

We also say that $P_\alpha$ is obtained from $P_\beta$ by \textbf{rounding a corner and locally losing one} $h$, if $P_\alpha$ is obtained from $P_\beta$ by a corner rounding such that the following conditions are satisfied:

\begin{enumerate}[(i)]
    \item Let $k$ denote the number of edges in $P_\beta$, with an endpoint at the rounded corner, which are labeled $h$. We requiere that $k>0$, so $k=1$ or $k=2$.
    \item Of the new in $P_\alpha$, created by the corner rounding operation, exactly $k-1$ are labelled $h$.
\end{enumerate}

Since we are interested in having an operation over $P_\alpha$ we will say that $P_\beta$ is obtained by $P_\beta$ by \textbf{corounding the corner} and \textbf{locally gaining one }$h$ if $P_\alpha$ is obtained from $P_\beta$ by \textbf{rounding a corner and locally losing one} $h$. If the context is clear we will just say that $P_\beta$ is obtained from $P_\beta$ by corounding the corner. 

\begin{figure}
    \centering
    \input{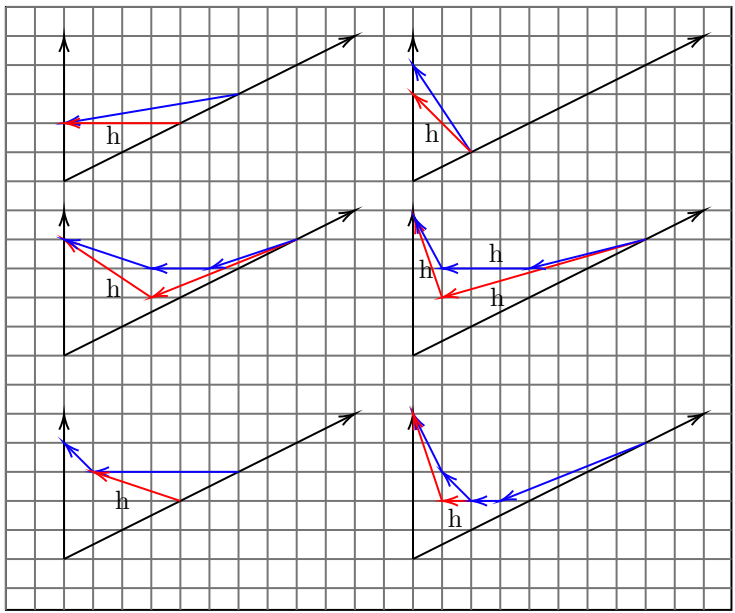}
    \caption{Examples of corounding the corner in $L(2,1)$. The paths $P_\alpha$ and $P_\beta$ are represented by blue and red respectively.}
    \label{Coroundingexamples}
\end{figure}

\begin{pro}
\label{equicomgeo}
For each $\epsilon>0$ and $k_0$ integer, there exists $L_\epsilon$ a good perturbation $\lambda_{\epsilon}$ of the contact form on $X_\Omega$ such that
\begin{enumerate}
    \item For every $k\leq k_0$ we have that $\text{ECC}_{k_0}^{L_\epsilon} (Y,\lambda_\epsilon)$ does not contain the orbits $e_1$ and $e_2$. More precisely, no orbit set $\alpha \in \text{ECC}_{k_0}^{L_\epsilon} (Y,\lambda_\epsilon) $ contains $e_1$ or $e_2$ with any multiplicity.
    \item Every orbit with period less than $L$ is non-degenerate.
    \item The map $\alpha\mapsto P_\alpha$ is a bijection between the generators of $\text{ECC}_{k_0}^{L_\epsilon} (Y,\lambda_\epsilon)$ and $\{P_\alpha:I(P_\alpha)=k\}$. Futhermore $I(\alpha)=I(P_\alpha)$ and $|\mathcal{A}(\alpha)-l_\Omega(P_\alpha)|<1/L$. 
    \item $\langle\partial\alpha,\beta\rangle=1$ if and only if $P_\beta$ is obtained from $P_\alpha$ by corounding the corner. 
\end{enumerate}
\end{pro}

See examples for $L(2,1)$ in figure \ref{Coroundingexamples}.

\begin{defi}
    \label{goodper}
    A perturbation of $\lambda_a$ for which the conditions 1., 2. and 3. holds is called a \textit{good perturbation}. 
\end{defi}

The following three subsections are dedicated to the proof of Proposition \ref{equicomgeo}. 

\subsection{Computations of the index}

Let $L>0$ and assume that $\Tilde{\lambda}_a^L$ is a Morse-Bott perturbation. In this section we prove parts $1., 2. $ and $3.$ of Proposition \ref{equicomgeo} by computing the $\text{ECH}$ index of an orbit set $\alpha=\{(\alpha_i,m_i)\}$. We also compute the Fredholm Index. 

\subsubsection{$\text{ECH}$ index}
\label{ECHindexconcave}

In what it follows we fix a trivialization over the contact structures of the Reeb orbits $e_+$ and $e_-$, that is, we fixed a vector $(v_1,v_2)\in \mathbb{Z}$ such that $(p,q)\times (v_1,v_2)=1$ and we choose the usual trivialization for $e_-$, see Lemma \ref{rotnumbers}. 

Before, we do the calculations we need to construct an auxiliary path $\bar{P}_\alpha$. Let $\alpha=\{(e_1,m_1)\}\cup\{(\alpha_i,m_i)\}\cup\{(e_2,m_2)\}$. Let $\alpha'=\{(\alpha_i,m_i)\}$, in subsection \ref{generators} we explained the path $P_\alpha$. By the homological conditions there exist unique $k$ and $l$ integers such that

$$m_1(v_1,v_2)+[\bar{P}_\alpha]+m_2(-1,0)=k(0,1)+l(n,m)$$

This define an unique polygonal path $\bar{P}_\alpha$. Notice that $\bar{P}_\alpha$ depends on the trivialization over $\{(e_+,m_+)\}$ and it is not a concave (nor convex) path. See figure \ref{PPalpha} for an example. 

\begin{figure}
    \centering
    \tikzset{every picture/.style={line width=0.75pt}} 

\begin{tikzpicture}[x=0.55pt,y=0.55pt,yscale=-1,xscale=1]

\draw  [draw opacity=0] (20,8) -- (309,8) -- (309,198) -- (20,198) -- cycle ; \draw  [color={rgb, 255:red, 175; green, 175; blue, 175 }  ,draw opacity=1 ] (20,8) -- (20,198)(40,8) -- (40,198)(60,8) -- (60,198)(80,8) -- (80,198)(100,8) -- (100,198)(120,8) -- (120,198)(140,8) -- (140,198)(160,8) -- (160,198)(180,8) -- (180,198)(200,8) -- (200,198)(220,8) -- (220,198)(240,8) -- (240,198)(260,8) -- (260,198)(280,8) -- (280,198)(300,8) -- (300,198) ; \draw  [color={rgb, 255:red, 175; green, 175; blue, 175 }  ,draw opacity=1 ] (20,8) -- (309,8)(20,28) -- (309,28)(20,48) -- (309,48)(20,68) -- (309,68)(20,88) -- (309,88)(20,108) -- (309,108)(20,128) -- (309,128)(20,148) -- (309,148)(20,168) -- (309,168)(20,188) -- (309,188) ; \draw  [color={rgb, 255:red, 175; green, 175; blue, 175 }  ,draw opacity=1 ]  ;
\draw    (60,188) -- (60,30) ;
\draw [shift={(60,28)}, rotate = 90] [color={rgb, 255:red, 0; green, 0; blue, 0 }  ][line width=0.75]    (10.93,-3.29) .. controls (6.95,-1.4) and (3.31,-0.3) .. (0,0) .. controls (3.31,0.3) and (6.95,1.4) .. (10.93,3.29)   ;
\draw    (60,188) -- (298.34,29.11) ;
\draw [shift={(300,28)}, rotate = 146.31] [color={rgb, 255:red, 0; green, 0; blue, 0 }  ][line width=0.75]    (10.93,-3.29) .. controls (6.95,-1.4) and (3.31,-0.3) .. (0,0) .. controls (3.31,0.3) and (6.95,1.4) .. (10.93,3.29)   ;
\draw [color={rgb, 255:red, 253; green, 0; blue, 0 }  ,draw opacity=1 ]   (100,48) -- (62,48) ;
\draw [shift={(60,48)}, rotate = 360] [color={rgb, 255:red, 253; green, 0; blue, 0 }  ,draw opacity=1 ][line width=0.75]    (10.93,-3.29) .. controls (6.95,-1.4) and (3.31,-0.3) .. (0,0) .. controls (3.31,0.3) and (6.95,1.4) .. (10.93,3.29)   ;
\draw [color={rgb, 255:red, 253; green, 0; blue, 0 }  ,draw opacity=1 ]   (300,28) -- (181.79,87.11) ;
\draw [shift={(180,88)}, rotate = 333.43] [color={rgb, 255:red, 253; green, 0; blue, 0 }  ,draw opacity=1 ][line width=0.75]    (10.93,-3.29) .. controls (6.95,-1.4) and (3.31,-0.3) .. (0,0) .. controls (3.31,0.3) and (6.95,1.4) .. (10.93,3.29)   ;
\draw [color={rgb, 255:red, 39; green, 0; blue, 255 }  ,draw opacity=1 ]   (160,88) -- (101.66,49.11) ;
\draw [shift={(100,48)}, rotate = 33.69] [color={rgb, 255:red, 39; green, 0; blue, 255 }  ,draw opacity=1 ][line width=0.75]    (10.93,-3.29) .. controls (6.95,-1.4) and (3.31,-0.3) .. (0,0) .. controls (3.31,0.3) and (6.95,1.4) .. (10.93,3.29)   ;
\draw [color={rgb, 255:red, 0; green, 0; blue, 252 }  ,draw opacity=1 ]   (180,88) -- (162,88) ;
\draw [shift={(160,88)}, rotate = 360] [color={rgb, 255:red, 0; green, 0; blue, 252 }  ,draw opacity=1 ][line width=0.75]    (10.93,-3.29) .. controls (6.95,-1.4) and (3.31,-0.3) .. (0,0) .. controls (3.31,0.3) and (6.95,1.4) .. (10.93,3.29)   ;
\end{tikzpicture}
    \caption{An example of a path $\bar{P}_\alpha$ where $\alpha=\{e_+^3,e_{(-1,0)},e_{(-2,3)},e_-^2\}$. In red arrows we represented the part of the path corresponding to $e_+$ and $e_-$. The rest of the path is in blue.}
    \label{PPalpha}
\end{figure}
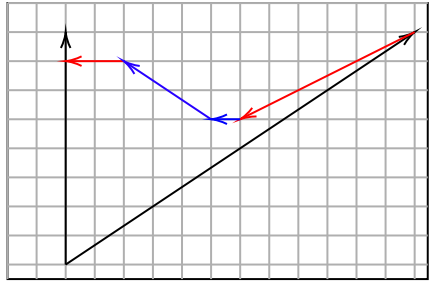

\begin{lemm}
    \label{explicitindex}
    Consider a good perturbation $\bar{\lambda}_a$ of $\lambda_a$ as in section \ref{Morse-Bott}.
     Let $\alpha=\{(e_+,m_+)\}\cup\{(e_-,m_-)\}\cup\{(\alpha_i,m_i)\}$ be a Reeb orbit set such that $[\alpha]=0\in H_1(Y)$. Then $I(\alpha,Z)$ does not depend on $Z\in H_2(\alpha,\emptyset,Z)$, futhermore
     \begin{enumerate}       
        
         \item (\textit{Relative Chern Class}) $C_\tau(\alpha)=c_1+c_2$ where $c_1\in \mathbb{Z}$ is the maximal integer such that $c_1(m,n)$ is contained in $\bar{P}_\alpha$. Analogously $c_2\in \mathbb{Z}$ is the maximal integer such that $c_2(0,1)$ is contained in $\bar{P}_\alpha$.

         \item (\textit{Relative self-intersection}) $Q_{\tau}(\alpha)=2A(\bar{P}_\alpha)$ where $A(\bar{P}_\alpha)$ denotes the area of the region defined by $\bar{P}_\alpha$ and the axis.

        \item (\textit{Conley-Zehnder Number}) Denote by $e$ the total number of elliptic orbits in $\{(\alpha_i,m_i)\}$. Then

         \begin{equation}
         \label{CZeq}
         CZ(\alpha)=-e+n_++n_-+ 2\displaystyle\sum_{i=1}^{n_+}\left\lfloor i\displaystyle\frac{a'(0)\times (v_1,v_2)}{a'(0)\times (p,q)} \right\rfloor+2\displaystyle\sum_{i=1}^{n_-}\left\lfloor -j \displaystyle\frac{a_2'(1)}{a_1(1)} \right\rfloor
         \end{equation}
         
    \end{enumerate}
        
\end{lemm}

\begin{proof}
Since $H_2(Y)=0$ the ECH index does not depend on the the relative homology $Z\in H_2(\alpha,\emptyset,Y)$. Now we construct a surface $S$ such that $\partial S=\alpha$.

The construction of the $S$ is a very classical argument and can be found in different forms in \cite{Choi_2014,cristofaro2020proof,hutchings2014lecture} and others. Here we modify that construction to fit our case. 
  We can construct a surface $S$ in $[-1,1]\times Y$ such that $[S]\in H_2(\alpha,\beta,Y)$. Then we use this manifold to compute $Q_\tau$ and $C_\tau$.
        
    \textit{Construction of the surface $S$:} 
        Consider the projections $\pi:[0,1]\times \To^2\rightarrow L(n,m)$ and consider the natural lifts of the orbits $\alpha_i$, for the orbit $e_+$ choose any orbit $e_+'$ in $\{0\}\times \To^2$ with homology $(n,m)$ and $e_-$ choose any orbit $e_-'$ in $\{0\}\times \To^2$ with homology $(0,1)$. Denote by $0=x_+<x_1<\cdots <x_M<x_-=1$  with $M=\sum m_i$ and each $x_i$ represent the point ${x_i}\times \To^2$ at which $\alpha_i$ appears.

        We construct this surface in three diferent steps.
        
        \textbf{Step 1:} \textit{Disjoint Cylinders.} We now describe a construction of disjoint cylinders $\mathcal{C}$. At level $\{1\}\times [0,1]\times \To^2$ we realize the following procedure: for each $\alpha_i$ with multiplicity $m_i$  to obtained a family of trivial cylinders in $\mathbb{R}\times [0,1]\times \To^2$. Choose $m_i$ points $x_{i1},\dots,x_{im_i}$ in a small neighborhood of $x_i$ and not containing any other $x_j$ with $i\not=j$. For each $x_{ik}$ choose an orbit with homology $[\alpha_i$] disjoint from all the others. For the case $e_+$ we make the perturbation slightly toward the right, respectively for the case of $e_-$ we make the perturbation slightly toward the left. 
        By following the $s$ direction downwards up to $\{0\}\times \{0\}\times \To^2$ we obtained a set $\mathcal{C}_1$ of disjoint cylinders. 
                
        \textbf{Step 2:} 
        \textit{Construction of the surface $S'$.}
        By the homological conditions we have that 
        $$[\alpha]=c_1(n,m)+c_2(0,1)$$
        Begin with $c_1$ disjoint orbits with homology $(n,m)$ in $\{0\}\times \{0\}\times \To^2$ move this orbits in the $x$ direction forming horizontal cilinders. Each time these cylinders encounter a vertical cilinder we realize negative surgeries similar to \cite{hutchings2005periodic}, in that way we resolve the singularities. After crossing every vertical cylinder we have $c_2$ cylinders in the $x$ direction with homology $(0,1)$. To end this step we make a slighty perturbation on ${1}\times[0,1]\times \To^2$ such that $\alpha$ is the boundary on ${1}\times[0,1]\times \To^2$ of the obtained surface $S'$.

        \textbf{Step 3:} \textit{Projecting the surface $S'$ to obtain $S$.} Consider the projection of $S$' by the quotient map of $\pi:\mathbb{R}\times [0,1]\times \To^2\rightarrow \mathbb{R}\times L(n,m)$. Note that the $c_1$ cilinders in the $x$ direction with homology $(n,m)$ colapses into disks. Similarly, the $c_2$ cylinders in the $x$ direction with homology $(0,1)$ colapses into disks. 

        This ends the construction of the surface $S$. 
        
        We now use this surface to compute $C_\tau$ and $Q_\tau$. See figure \ref{TheSurface} for a schematic picture of the surface $S'$ when projected into $[-1,1]\times [0,1]$.

        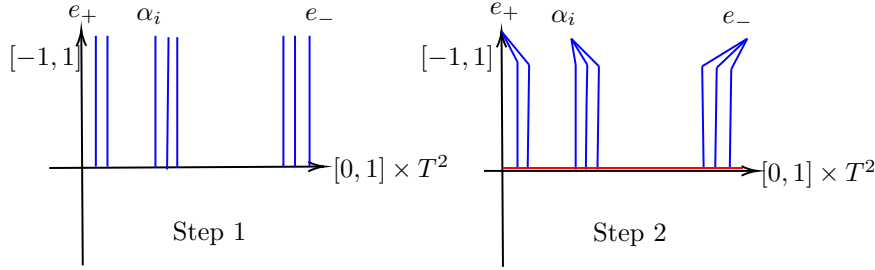
\begin{figure}
            \centering
            \tikzset{every picture/.style={line width=0.75pt}} 

\begin{tikzpicture}[x=0.55pt,y=0.55pt,yscale=-1,xscale=1]

\draw    (39,112) -- (226,111.01) ;
\draw [shift={(228,111)}, rotate = 179.7] [color={rgb, 255:red, 0; green, 0; blue, 0 }  ][line width=0.75]    (10.93,-3.29) .. controls (6.95,-1.4) and (3.31,-0.3) .. (0,0) .. controls (3.31,0.3) and (6.95,1.4) .. (10.93,3.29)   ;
\draw    (62,179) -- (61.01,19) ;
\draw [shift={(61,17)}, rotate = 89.65] [color={rgb, 255:red, 0; green, 0; blue, 0 }  ][line width=0.75]    (10.93,-3.29) .. controls (6.95,-1.4) and (3.31,-0.3) .. (0,0) .. controls (3.31,0.3) and (6.95,1.4) .. (10.93,3.29)   ;
\draw [color={rgb, 255:red, 4; green, 4; blue, 253 }  ,draw opacity=1 ]   (79,111) -- (79,21) ;
\draw [color={rgb, 255:red, 4; green, 4; blue, 253 }  ,draw opacity=1 ]   (71,111) -- (71,21) ;
\draw [color={rgb, 255:red, 4; green, 4; blue, 253 }  ,draw opacity=1 ]   (112,112) -- (112,21) ;
\draw [color={rgb, 255:red, 4; green, 4; blue, 253 }  ,draw opacity=1 ]   (120,113) -- (121,22) ;
\draw [color={rgb, 255:red, 4; green, 4; blue, 253 }  ,draw opacity=1 ]   (127,112) -- (127,22) ;
\draw [color={rgb, 255:red, 4; green, 4; blue, 253 }  ,draw opacity=1 ]   (218,111) -- (218,21) ;
\draw [color={rgb, 255:red, 4; green, 4; blue, 253 }  ,draw opacity=1 ]   (208,111) -- (208,21) ;
\draw [color={rgb, 255:red, 4; green, 4; blue, 253 }  ,draw opacity=1 ]   (200,111) -- (200,21) ;
\draw    (338,114) -- (524,114) ;
\draw [shift={(526,114)}, rotate = 180] [color={rgb, 255:red, 0; green, 0; blue, 0 }  ][line width=0.75]    (10.93,-3.29) .. controls (6.95,-1.4) and (3.31,-0.3) .. (0,0) .. controls (3.31,0.3) and (6.95,1.4) .. (10.93,3.29)   ;
\draw    (351,180) -- (350.01,20) ;
\draw [shift={(350,18)}, rotate = 89.65] [color={rgb, 255:red, 0; green, 0; blue, 0 }  ][line width=0.75]    (10.93,-3.29) .. controls (6.95,-1.4) and (3.31,-0.3) .. (0,0) .. controls (3.31,0.3) and (6.95,1.4) .. (10.93,3.29)   ;
\draw [color={rgb, 255:red, 4; green, 4; blue, 253 }  ,draw opacity=1 ]   (368,112) -- (369,41) ;
\draw [color={rgb, 255:red, 4; green, 4; blue, 253 }  ,draw opacity=1 ]   (361,112) -- (361,39) ;
\draw [color={rgb, 255:red, 4; green, 4; blue, 253 }  ,draw opacity=1 ]   (401,113) -- (401,41) ;
\draw [color={rgb, 255:red, 4; green, 4; blue, 253 }  ,draw opacity=1 ]   (408,113) -- (408.53,64.99) -- (409,41) ;
\draw [color={rgb, 255:red, 4; green, 4; blue, 253 }  ,draw opacity=1 ]   (416,113) -- (417,42) ;
\draw [color={rgb, 255:red, 4; green, 4; blue, 253 }  ,draw opacity=1 ]   (507,112) -- (508,44) ;
\draw [color={rgb, 255:red, 4; green, 4; blue, 253 }  ,draw opacity=1 ]   (497,112) -- (498,43) ;
\draw [color={rgb, 255:red, 4; green, 4; blue, 253 }  ,draw opacity=1 ]   (489,112) -- (488,42) ;
\draw [color={rgb, 255:red, 255; green, 0; blue, 0 }  ,draw opacity=1 ]   (351,112) -- (516,112) ;
\draw [color={rgb, 255:red, 4; green, 4; blue, 253 }  ,draw opacity=1 ]   (361,39) -- (350,18) ;
\draw [color={rgb, 255:red, 4; green, 4; blue, 253 }  ,draw opacity=1 ]   (369,41) -- (350,18) ;
\draw [color={rgb, 255:red, 4; green, 4; blue, 253 }  ,draw opacity=1 ]   (401,41) -- (398,23) ;
\draw [color={rgb, 255:red, 4; green, 4; blue, 253 }  ,draw opacity=1 ]   (409,41) -- (398,23) ;
\draw [color={rgb, 255:red, 4; green, 4; blue, 253 }  ,draw opacity=1 ]   (417,42) -- (398,23) ;
\draw [color={rgb, 255:red, 4; green, 4; blue, 253 }  ,draw opacity=1 ]   (508,44) -- (519,23) ;
\draw [color={rgb, 255:red, 4; green, 4; blue, 253 }  ,draw opacity=1 ]   (498,43) -- (519,23) ;
\draw [color={rgb, 255:red, 4; green, 4; blue, 253 }  ,draw opacity=1 ]   (488,42) -- (519,23) ;

\draw (232,101.4) node [anchor=north west][inner sep=0.75pt]    {$[ 0,1] \times T^{2}$};
\draw (122,147) node [anchor=north west][inner sep=0.75pt]   [align=left] {Step 1};
\draw (526,104.4) node [anchor=north west][inner sep=0.75pt]    {$[ 0,1] \times T^{2}$};
\draw (411,148) node [anchor=north west][inner sep=0.75pt]   [align=left] {Step 2};
\draw (9,25.4) node [anchor=north west][inner sep=0.75pt]    {$[ -1,1]$};
\draw (292,25.4) node [anchor=north west][inner sep=0.75pt]    {$[ -1,1]$};
\draw (50,-1.6) node [anchor=north west][inner sep=0.75pt]    {$e_{+}$};
\draw (213,0.4) node [anchor=north west][inner sep=0.75pt]    {$e_{-}$};
\draw (97,0.4) node [anchor=north west][inner sep=0.75pt]    {$\alpha _{i}$};
\draw (340,-3.6) node [anchor=north west][inner sep=0.75pt]    {$e_{+}$};
\draw (382,1.4) node [anchor=north west][inner sep=0.75pt]    {$\alpha _{i}$};
\draw (500,2.4) node [anchor=north west][inner sep=0.75pt]    {$e_{-}$};

\end{tikzpicture}
            \caption{Schematic representation of the surface $C$ and $S'$ in lemma \ref{explicitindex}. In this case $\alpha=\{(e_+,2),(\alpha_i,3),(e_-,2)\}$.}
            \label{TheSurface}
        \end{figure}

\begin{enumerate}
    \item (\textit{Relative Chern Class}) 
    We need a a generic section on $\xi|_S$ constant under $\tau$. Denote by $\eta$ the vector field over $S$ defined as $x(1-x)\partial_x$. It is easy to check that $\eta$ is a generic vector field constant under $\tau$. Notice that $\eta$ is $0$ exactly on the disks on $S$ obtained by the quotient that transform $S'$ into $S$. Therefore, $c_\tau(S)=\#\eta^{-1}(0)=c_1+c_2$.
    
    \item (\textit{Relative self-intersection}) 
        To calculate $Q_\tau(S)$ we use the expression given in equation $\eqref{relativeQtau}$ 
            $$Q_\tau(S)=c_1(NS,\tau)+w_\tau(S)$$
        By construction $w_\tau(S)=0$. Let $\phi$ be the field obtained by projecting $\partial_s+\partial_x$ into $S$. Notice that $\phi$ is $0$ exactly on the surgery points obtained by resolving singularities in step $2$. When resolving the singularites we obtained a number of zeros equal a determinant $\phi$ given by the resolution of the singularities. Since the determinant can be interpret as and area, a carefully organization of the terms will lead us to conclude that $\#\phi^{-1}(0)=2 A(P_\alpha')$.
            
       \item (\textit{Conley-Zehnder Number}) Notice that equation \eqref{CZeq} follows directly from lemma \ref{CZelliptic} and the assumptions in subsection \ref{Morse-Bott}. 
\end{enumerate}
\end{proof}

With this calculation in place we can conclude the proof of the part (b) of proposition \ref{equicomgeo}.

\begin{proof}{of parts 1., 2. and 3. of proposition \ref{equicomgeo}}
    
    We begin by proving 3. Suppose that we have a Morse-Bott perturbation of $\lambda_a$. Using Pick's theorem we can writte     $$Q_\tau(\alpha)=2\iota(\bar{P}_\alpha)+c_1+c_2+m+n+e+h-1$$
    Where $\iota(\bar{P}_\alpha)$ is the count of the interior points of the closed region defined by $\bar{P}_\alpha$ and the axes. Then 
    $$I(\alpha)=2\mathcal{L}(\bar{P}_\alpha)+h$$
    This proves that $I(\alpha)=I(P_\alpha)$. The action $\mathcal{A}$ and $l_\Omega$ are as close as we want in virtue of the Morse-Bott perturbation. A consequence of this calculation is that the set  $A_k=\{\alpha:I(\alpha)=k\}$ is finite. 

    It is easy to see that we can use the finitness of the sets $A_k$ to find $\epsilon>0$ and $L>0$ such that the perturbation $\lambda_a^{\epsilon,L}$ satisfies the properties 1., 2. and 3. of the proposition. 
\end{proof}

\subsubsection{Fredholm Index}    
    After the calculation of the ECH index we can easily deduce the Fredholm Index relevant for our case. In the following lemma we do not rule out $J$-holomorphic curves with genus different from zero. To rule out the non-zero curves we will make use of the results given in this as well as the result of sections \ref{Positivity} and \ref{Pathcannotcross}. See lemma \ref{genuszero}.
    
    \begin{lemm}[Combinatorial Fredholm Index]
    Suppose that $\alpha=\{(\alpha_i,m_i)\}$ and $\beta=\{(\beta_i,n_i)\}$ are admissible orbit sets and suppose that neither of $\alpha$ nor $\beta$ possesses the elliptic orbits $e_+$ or $e_-$ with any multiplicity. Let $C$ be $\mathcal{M}(\alpha,\beta)$ any irreducible $J$-holomorphic curve from $\alpha$ to $\beta$. Then
    \begin{equation}
    \label{FredInd}
        \text{ind}(C)=-2+2g+2 e_\beta+h+2 c_1+2 c_2
    \end{equation}
        Where $e_\beta$ denotes the number of elliptic orbits, $h$ is the total number of hyperbolic orbits and $c_1$ and $c_2$ denotes the difference between the Chern classes of $\alpha$ and $\beta$ calculated in lemma \ref{explicitindex}. 
    \end{lemm}

    \begin{proof}
        Recall from equation \eqref{FredholmIndex} that the Fredholm Index has the following form 
        \begin{align}
        \text{ind}(C)=-\chi(C)+2 c_\tau(C)+\displaystyle\sum_{i=1}^n \text{CZ}_\tau(\alpha_i)-\displaystyle\sum_{j=1}^m\text{CZ}_\tau(\beta_j)
        \end{align}
        The number of ends of the curve $C$ is given by the sum $e_\alpha+e_\beta+h$ where $e_\alpha$ denotes the ends at elliptic orbits of $\alpha$, similarly, $e_\beta$ denotes the ends at elliptic orbits of $\beta$. So the Euler Characteristic of $C$ is given by the formula 
        $$\chi(C)=2-2g-e_--e_+-h$$
    Also as explained in lemma \ref{explicitindex} we have that $c_\tau(C)=c_1^\alpha- c_1^\beta +c_1^\alpha-c_1^\beta$ which we have denoted by simply $c_1+c_2$. Summing all of this up we obtain equation \eqref{FredInd}.
    \end{proof}
\subsubsection{Embedded Contact Homology of Lens Spaces}
\label{ECH-of-Lens-Spaces}

Before continuing with the proof of Proposition \ref{equicomgeo} we use the calculation of section \ref{ECHindexconcave} to compute the embedded contact homology of the lens spaces when $\Gamma=0$. Here we use analogous arguments to the ones given in \cite[Sec.~3.7]{hutchings2014lecture} and in \cite[Theo.~7.6]{nelson2022embedded}. From these calculations we will deduce Lemma \ref{capacitiesellipsoids} and with making use of Proposition \ref{equicomgeo} we also prove Theorem \ref{concavecapacities}.

We begin by ilustrating the calculation of the ECH index for the ellipsoid with singularity of the form $M(n,1)$.

\begin{exa}[ECH index of the irrational ellipsoid with singularities] 
Suposse that $a$ and $b$ are numbers such that $a/b$ is irrational. We say that $E_n(a,b)$ is an \textit{irrational ellipsoid with singularities}. Analogous to the usual ellipsoid we have that for $E_n(a,b)$ there exist two elliptic orbits $\gamma_0$ and $\gamma_1$ with periods equal to $a$ and $b$. In this case the vector $(-1,0)$ define a trivialization of the contact structure of $\gamma_0$ and $(0,1)$ defines a trivialization over the orbit $\gamma_1$. With this trivialization the rotations numbers are equal $\phi_0=\displaystyle\frac{b-a}{nb}$  and $\phi_1=\displaystyle\frac{a-b}{na}$ for $\gamma_0$ and $\gamma_1$ respectively. 

Take $r+s=kn$ then using the Lemma \ref{explicitindex}  we can find that with this trivialization we have that
\begin{align}
    Q_\tau(\gamma_0^r\gamma_1^s)&=k^2p \\
    c_\tau(\gamma_0^r\gamma_1^s)&=2k
\end{align}

for $\gamma_0$ and $\gamma_1$ respectively. Therefore 
\begin{align}
\label{indexorbllipse}
I(\gamma_0^r\gamma_1^s)&=k^2p+2k+2\displaystyle\sum_{i=1}^r\left(\left\lfloor i\displaystyle\frac{b-a}{pb} \right\rfloor +1\right)+2\displaystyle\sum_{i=1}^s\left(\left\lfloor j \displaystyle\frac{a-b}{pa} \right\rfloor +1\right)    \\
&=k(k+1)p+2k+2\displaystyle\sum_{i=1}^r\left\lfloor i\displaystyle\frac{b-a}{pb} \right\rfloor+2\displaystyle\sum_{i=1}^s\left\lfloor j \displaystyle\frac{a-b}{pa} \right\rfloor
\end{align}

Which implies that $I(\gamma_0^r\gamma_1^s)$ is even. It follows from Proposition \ref{bijectieven} below that $I$ is a bijection with the even numbers.
\end{exa}

Now we prove the following central proposition.

\begin{pro}
\label{bijectieven}
Let $Y=\partial E_{(n,m)}(a,b)$ such that $b/a$ is irrational. The $\text{ECH}$ index $I$ is a bijective map betweeen the Reeb orbits sets of $Y$ and the even numbers. 
\end{pro}

\begin{proof} 
Take $(k_1,k_2)$ a lattice point contained in $V_{n,m}$. Let $\eta(k_1,k_2)$ be the number of lattice points contained in the region defined by the line $L$ with slope $b/a$ passing through that point and the axis of $V_{n,m}$. Since $b/a$ is irrational it is clear that $\eta$ is a bijective function from the lattice points in $V_{n,m}$ and the integers.

Take a vector $(v_1,v_2)$ in $\mathbb{Z}^2$ such that $(p,q)\times(v_1,v_2)=1$ pointing towards the interior of $V_{n,m}$.  Take $(r,s)$ a lattice point contained in $V_{n,m}$. It is easy to check that there exist exactly one $k$ and one $l$ such that $(k_1,k_2)=k(p,q)+r(v_1,v_2)$. Equivalently, we say that there exist exactly one $k$ and one $l$ such that  $k(p,q)+l(v_1,v_2)+r(-1,0)+s(0,-1)=0$ which implies that the Reeb orbits of $Y$  are in bijection with the lattice points of $V_{(n,m)}$. More precisely given $\gamma_1^{r}\gamma_2^s$ it defines a unique closed region with the axis and the integers $(k_1,k_2)$ is where the end of $r(v_1,v_2)$ and $s(-1,0)$ meet. 

Therefore, it is enough to prove that $I(\gamma_1^r\gamma_2^s)=2\eta(k_1.k_2)$. Futhermore, the argument in the above paragraph suggests that $\gamma_1$ can be indentified with $(v_1,v_2)$ and $\gamma_2$ can be identified $(-1,0)$.

Let $\gamma_1^r\gamma_2^s$ with homology equal $0$. Denote by $R$ be the region defined by the Reeb Orbits set and the axis. Notice that the region $R$ and the line $L$ naturally defined three subregions: the region $R_1$ defined by the axis and two of three lines $L$, the vector representing $\gamma_2^s$ the vector representing $\gamma_1^s$, the region $R_2$ defined by the $y$-axis and the line $L$ and the vector representing $\gamma_2^s$, and the region $R_3$ defined by the $(p,q)$-axis, the line $L$ and the vector representing $\gamma_1^r$. 

We examplify this situation in figure \ref{L(3,2) counting} for $L(3,2)$ and $(v_1,v_2)=(-2,1)$. Also notice that there are two other possible cases, in the case of this figure the line $L$ goes over the vector corresponding to $\gamma_2^s$ and under the vector corresponding to $\gamma_1^r$. All of them are similar so we suppose that we are in this same setting that is $L$ goes over the vector corresponding of $\gamma_1^s$ and under the vector corresponding to $\gamma_1^s$.

Notice that 
\begin{align}    
\eta(k_1,k_2)=\bar{\mathcal{L}}(R_1)+\bar{\mathcal{L}}(R_2)-s-2
\end{align}

where $\iota(R_i)$ is the number of interior lattice points of $R_i$ and $b(R_i)$  is the number of boundary lattice points.

    \begin{figure}
        \centering
       \tikzset{every picture/.style={line width=0.75pt}} 

\begin{tikzpicture}[x=0.50pt,y=0.50pt,yscale=-1,xscale=1]

\draw  [draw opacity=0][fill={rgb, 255:red, 233; green, 237; blue, 247 }  ,fill opacity=0.14 ] (-11,-0.5) -- (480.5,-0.5) -- (480.5,333.5) -- (-11,333.5) -- cycle ; \draw  [color={rgb, 255:red, 216; green, 216; blue, 216 }  ,draw opacity=1 ] (-11,-0.5) -- (-11,333.5)(25,-0.5) -- (25,333.5)(61,-0.5) -- (61,333.5)(97,-0.5) -- (97,333.5)(133,-0.5) -- (133,333.5)(169,-0.5) -- (169,333.5)(205,-0.5) -- (205,333.5)(241,-0.5) -- (241,333.5)(277,-0.5) -- (277,333.5)(313,-0.5) -- (313,333.5)(349,-0.5) -- (349,333.5)(385,-0.5) -- (385,333.5)(421,-0.5) -- (421,333.5)(457,-0.5) -- (457,333.5) ; \draw  [color={rgb, 255:red, 216; green, 216; blue, 216 }  ,draw opacity=1 ] (-11,-0.5) -- (480.5,-0.5)(-11,35.5) -- (480.5,35.5)(-11,71.5) -- (480.5,71.5)(-11,107.5) -- (480.5,107.5)(-11,143.5) -- (480.5,143.5)(-11,179.5) -- (480.5,179.5)(-11,215.5) -- (480.5,215.5)(-11,251.5) -- (480.5,251.5)(-11,287.5) -- (480.5,287.5)(-11,323.5) -- (480.5,323.5) ; \draw  [color={rgb, 255:red, 216; green, 216; blue, 216 }  ,draw opacity=1 ]  ;
\draw    (25,323.5) -- (25,37.5) ;
\draw [shift={(25,35.5)}, rotate = 90] [color={rgb, 255:red, 0; green, 0; blue, 0 }  ][line width=0.75]    (10.93,-3.29) .. controls (6.95,-1.4) and (3.31,-0.3) .. (0,0) .. controls (3.31,0.3) and (6.95,1.4) .. (10.93,3.29)   ;
\draw    (25,323.5) -- (455.34,36.61) ;
\draw [shift={(457,35.5)}, rotate = 146.31] [color={rgb, 255:red, 0; green, 0; blue, 0 }  ][line width=0.75]    (10.93,-3.29) .. controls (6.95,-1.4) and (3.31,-0.3) .. (0,0) .. controls (3.31,0.3) and (6.95,1.4) .. (10.93,3.29)   ;
\draw [color={rgb, 255:red, 253; green, 6; blue, 6 }  ,draw opacity=1 ]   (61,215.5) -- (27,215.5) ;
\draw [shift={(25,215.5)}, rotate = 360] [color={rgb, 255:red, 253; green, 6; blue, 6 }  ,draw opacity=1 ][line width=0.75]    (10.93,-3.29) .. controls (6.95,-1.4) and (3.31,-0.3) .. (0,0) .. controls (3.31,0.3) and (6.95,1.4) .. (10.93,3.29)   ;
\draw [color={rgb, 255:red, 253; green, 6; blue, 6 }  ,draw opacity=1 ]   (97,215.5) -- (63,215.5) ;
\draw [shift={(61,215.5)}, rotate = 360] [color={rgb, 255:red, 253; green, 6; blue, 6 }  ,draw opacity=1 ][line width=0.75]    (10.93,-3.29) .. controls (6.95,-1.4) and (3.31,-0.3) .. (0,0) .. controls (3.31,0.3) and (6.95,1.4) .. (10.93,3.29)   ;
\draw [color={rgb, 255:red, 255; green, 0; blue, 0 }  ,draw opacity=1 ]   (241,143.5) -- (170.79,178.61) ;
\draw [shift={(169,179.5)}, rotate = 333.43] [color={rgb, 255:red, 255; green, 0; blue, 0 }  ,draw opacity=1 ][line width=0.75]    (10.93,-3.29) .. controls (6.95,-1.4) and (3.31,-0.3) .. (0,0) .. controls (3.31,0.3) and (6.95,1.4) .. (10.93,3.29)   ;
\draw [color={rgb, 255:red, 255; green, 0; blue, 0 }  ,draw opacity=1 ]   (313,107.5) -- (242.79,142.61) ;
\draw [shift={(241,143.5)}, rotate = 333.43] [color={rgb, 255:red, 255; green, 0; blue, 0 }  ,draw opacity=1 ][line width=0.75]    (10.93,-3.29) .. controls (6.95,-1.4) and (3.31,-0.3) .. (0,0) .. controls (3.31,0.3) and (6.95,1.4) .. (10.93,3.29)   ;
\draw [color={rgb, 255:red, 255; green, 0; blue, 0 }  ,draw opacity=1 ]   (385,71.5) -- (314.79,106.61) ;
\draw [shift={(313,107.5)}, rotate = 333.43] [color={rgb, 255:red, 255; green, 0; blue, 0 }  ,draw opacity=1 ][line width=0.75]    (10.93,-3.29) .. controls (6.95,-1.4) and (3.31,-0.3) .. (0,0) .. controls (3.31,0.3) and (6.95,1.4) .. (10.93,3.29)   ;
\draw [color={rgb, 255:red, 255; green, 0; blue, 0 }  ,draw opacity=1 ]   (457,35.5) -- (386.79,70.61) ;
\draw [shift={(385,71.5)}, rotate = 333.43] [color={rgb, 255:red, 255; green, 0; blue, 0 }  ,draw opacity=1 ][line width=0.75]    (10.93,-3.29) .. controls (6.95,-1.4) and (3.31,-0.3) .. (0,0) .. controls (3.31,0.3) and (6.95,1.4) .. (10.93,3.29)   ;
\draw [color={rgb, 255:red, 255; green, 0; blue, 0 }  ,draw opacity=1 ]   (169,179.5) -- (98.79,214.61) ;
\draw [shift={(97,215.5)}, rotate = 333.43] [color={rgb, 255:red, 255; green, 0; blue, 0 }  ,draw opacity=1 ][line width=0.75]    (10.93,-3.29) .. controls (6.95,-1.4) and (3.31,-0.3) .. (0,0) .. controls (3.31,0.3) and (6.95,1.4) .. (10.93,3.29)   ;
\draw [color={rgb, 255:red, 3; green, 27; blue, 255 }  ,draw opacity=1 ][fill={rgb, 255:red, 138; green, 65; blue, 65 }  ,fill opacity=1 ]   (20,126) -- (192.5,313) ;
\draw  [color={rgb, 255:red, 151; green, 171; blue, 235 }  ,draw opacity=0.26 ][line width=3] [line join = round][line cap = round] (172.5,87) .. controls (172.5,87) and (172.5,87) .. (172.5,87) ;

\draw (81,218.5) node [anchor=north west][inner sep=0.75pt]   [align=left] {{\scriptsize (2,3)}};
\draw (39,181.4) node [anchor=north west][inner sep=0.75pt]    {$R_{2}$};
\draw (49,243.4) node [anchor=north west][inner sep=0.75pt]    {$R_{1}$};
\draw (171,190.4) node [anchor=north west][inner sep=0.75pt]    {$R_{3}$};
\draw (80,163) node [anchor=north west][inner sep=0.75pt]  [font=\large] [align=left] {L};

\end{tikzpicture}    
        \caption{Example for $L(3,2)$ of proposition \ref{bijectieven}.}
        \label{L(3,2) counting}
    \end{figure}
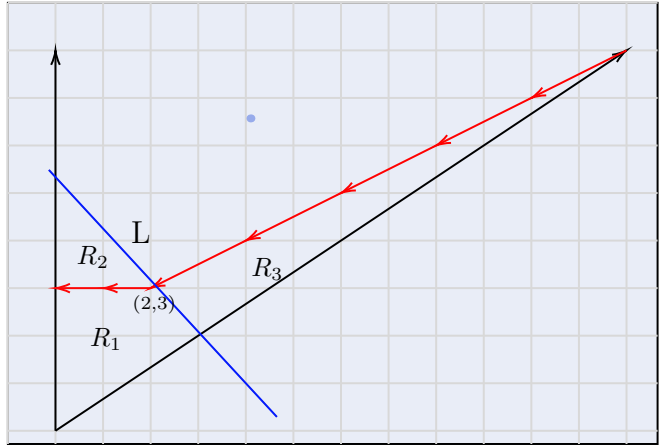

As explained in Lemma \ref{rotnumbers} the vector $(v_1,v_2)$ induced a triviazalition over the contact structure restricted to $\gamma_1$. By the same lemma we also have that with this trivialization
\begin{align}
Q_\tau(\gamma^r\gamma^s)&=2 A(R) \\    
C_\tau(\gamma^r\gamma^s)&=|k|+|k'| 
\end{align}
  Here $A(R)$ is the area of the region $R$, and, $k$ and $k'$ are integers such that $r(v_1,v_2)+s(-1,0)=k(p,q)+k'(0,1)$.   
We also have,
\begin{align}
    CZ_\tau(\gamma^r\gamma^s)&=r+s+ 2\displaystyle\sum_{i=1}^r\left\lfloor i\displaystyle\frac{(a,b)\times (v_1,v_2)}{(a,b)\times (p,q)} \right\rfloor+2\displaystyle\sum_{i=1}^s\left\lfloor j \displaystyle\frac{b}{a} \right\rfloor
\end{align}

By Pick's theorem $$Q_\tau(\gamma_1^r\gamma_2^s)=2\iota(R_1)+2\iota(R_3)+b(R_1)+b(R_3)-2$$

Also $b(R_1)+b(R_3)=|k|+|k'|+r+s$. Therefore, 
\begin{align*}I(\gamma_1^r\gamma_2^s)=
    &2\left(\iota(R_1)+\iota(R_3)+b(R_1)+b(R_3)-1\right)\\
    +&2\left(\displaystyle\sum_{i=1}^r\left\lfloor i\displaystyle\frac{(a,b)\times (v_1,v_2)}{(a,b)\times (p,q)} \right\rfloor+\displaystyle\sum_{i=1}^s\left\lfloor j \displaystyle\frac{b}{a} \right\rfloor\right)
    \end{align*}  
    It is easy to see that $\sum_{i=1}^s\left\lfloor j \displaystyle\frac{b}{a} \right\rfloor$ correspond to $\mathcal{L}(R_2)-s-1$. Therefore we can rewrite the above equation as 
    \begin{align*}I(\gamma_1^r\gamma_2^s)=&
    2\left(\mathcal{L}(R_1)+\mathcal{L}(R_3)+\mathcal{L}(R_2)-s-2\right)\\
    +&2\left(\displaystyle\sum_{i=1}^r\left\lfloor i\displaystyle\frac{(a,b)\times (v_1,v_2)}{(a,b)\times (p,q)} \right\rfloor\right)
    \end{align*}  
 So conclude the proof, it is enough to see that $\sum_{i=1}^r\left\lfloor i\frac{(a,b)\times (v_1,v_2)}{(a,b)\times (p,q)} \right\rfloor=-\mathcal{L}(R_3)$. We can check this by noticing that this sum correspond to minus the sum of points of the region $R_3'$ obtained by multiplying the region $R_3$ by the matrix defined in equation (\ref{trivmatrix}). Since this matrix is $\text{SL}_2(\mathbb{Z})$ the numbers of points in $R_3$ and $R_3'$ remains the same. 

 This concludes the proof.
\end{proof}

\begin{col} Suppose that $Y$ is a lens space. Then
\label{ECHlenscol}
\begin{align}
\label{ECHlens}
\text{ECH}_*(Y,\emptyset) = \left\{ 
    \begin{array}{lcc}
                \mathbb{Z}_2     &   \text{if } &  *=2k \\
                       0         &   \text{if } &  *=2k+1
             \end{array}
   \right.
\end{align}    
\end{col}

\begin{proof}
This follows directly from the fact that the differential is an operator of index $-1$. By proposition \ref{bijectieven} it follows that the differential is equal to zero and therefore the homology correspond to the one given in equation \ref{ECHlens}. The result follows from the topological invariance of the embedded contact homology. 
\end{proof}

Now we can prove Lemma \ref{capacitiesellipsoids} where we stablish the capacities for the ellipsoids with singularity.

\begin{proof}[proof of Lemma \ref{capacitiesellipsoids}]
    Suppose that $a/b$ is irrational. Notice that  $\mathcal{A}(e_1^{k_1}e_2^{k_2})=k_1a+k_2b$ then it follows from Corollary \ref{ECHlenscol} that the capacities are indeed a reorganization of the numbers $k_1a+k_2b$ with an additional homology condition. Choose a primitive vector $(v_1,v_2)$ such that $(v_1,v_2)\times (m,n)=1$. Identify $e_1$ with $(v_1,v_2)\in \mathbb{Z}^2$ and $e_2$ with $(0,1)\in \mathbb{Z}^2$. Then the condition $[e_1^{k_1}e_2^{k_2}]=0\in H_1(\partial E_{m,n}(a,b))$ becomes
    $$k_1(v_1,v_2)+k_2(-1,0)=s(m,n)+l(0,1)$$
    from which we deduce the relationship $k_1+nk_2=lm$. The result follows. 
\end{proof}

Using the Proposition \ref{equicomgeo} we end this section with the proof of Theorem \ref{concavecapacities} where we stablish the capacities for any concave domain in $M(n,m)$.

\begin{proof}[proof of Theorem \ref{concavecapacities}]    
    Let $a=(a_1,a_2):[0,1]\rightarrow V_{n,m}$ be the function such that $a[0,1]=\partial \Omega$ and $\lambda_a=a_1\dd t_1+a_2\dd t_2$. We aproximate the function $a$ by a family of concave smooth functions $a_\epsilon:[0,1]\rightarrow V_{n,m}$ such that $a_\epsilon(x)=a(x)$ if $[\epsilon,1-\epsilon]$, $a'_\epsilon(0)\rightarrow u_0$ and  $a'_\epsilon(1)\rightarrow u_1$ when $\epsilon\rightarrow 0$. This construction and lemma \ref{rotnumbers} implies that the ECH index of $e_+$ and $e_-$ tends to infinity as $\epsilon\rightarrow 0$. By the continuity of the spectrum with respect $a$ (see lemma 2.3 and 2.4 of \cite{Choi_2014}) it follows that it is enough to consider generators that does not contained multiples of $e_+$ or $e_-$.
    
    Notice that by Lemma \ref{capsimp} and Corollary \ref{ECHlenscol} it is enough to prove that the sum of all concave generators with elliptic label is the only closed, non-exact and minimal sum of generators. 

    Suppose that $\Lambda=\alpha_1+\cdots+\alpha_r$ is the sum of all the concave elliptic generators. Take $1\leq i\leq r$ and consider $\partial \alpha_i=\alpha_i^1+\cdots +\alpha_i^s$. Notice that for $\alpha_i^j$ with $0\leq j \leq s$ is a generator with index $2k-1$ and exactly one $'\hat{h}'$ label. By definittion there exist exactly one $\alpha_{i'}$ with $1\leq i'\leq r$ and $i'\not =i$ such that  $P_{\alpha_{i'}}$ that is obtained by corounding a corner of $P_{\alpha_i^j}$. Therefore $\Lambda$ is closed. 
    The sum of generators $\Lambda$ is no null-homologous because the differential by definition increases the number of $'h'$ labels by one. 
    
    Finally, we prove that $\Lambda$ is minimal. Suppose that $\Lambda=\Lambda'+\Lambda''$ such that $\partial \Lambda''=0$ therefore $\partial \Lambda'=0$. Then without loses of generality we can suppose that $\Lambda''$ is exact. Since $\Lambda$ consist only of concave elliptic generators this is not possible. So $\Lambda$ is minimal. An analogous argument proves that $\Lambda$ is the only non-nullhomologous sum of generators. 

\end{proof}

\subsection{Positivity}
\label{Positivity}
As usual we suppose that we are working with a good perturbation of $\lambda_a$ (see Definition \ref{goodper}).

In this section we prove an important property that impose major restrictions on what kind of $J$-holomorphic curves can appear. 
The lemma and its proof are inspired by \cite{keon}. This is a consequences of the intersection positivity of $J$-holomorphic curves in four dimensions. 

It is interesting to note that this property has appeared in different ocassions in the literature, see for example \cite{keon,cristofaro2020proof,hutchings2006rounding}. Sometimes this property is also called \textit{local energy inequality} as in \cite{hutchings2005periodic,yao2022computing}.

Let $\alpha$ and $\beta$ be a couple of Reeb sets that do not possesses the orbits $e_+$ and $e_-$. Suppose that $S$ is a surface with an homology class in $H_2(Y,\alpha,\beta)$. Notice that for $x\in (0,1)$ we can define an intersection $S_x=S\cap \mathbb{R}\times\{x\}\times \To^2$. 

Notice that the orientation of $S_x$ is induce by the orientation of $S$. Our convention (which follows the convention of \cite[sec~3.4]{hutchings2005periodic}) is that we take the opposite of the usual `outer normal first' convention. Therefore we get a well-defined class $[S_x]\in H_1(\To^2)$ which is called the \textit{slice class}.

\begin{lemm}
    Suppose that $C\in \mathcal{M}(\alpha,\beta)$ then for every $x$ we have that 
    \begin{equation}
        \label{positivity}
        a'(x)\times [C_x]\geq 0
    \end{equation}
    with equality if and only if $C_x=0$.
\end{lemm}
Here, $(a,b)\times (c,d)$, where $(a,b),(c,d)\in \mathbb{R}^2$, is defined to be the quantity $ad-bc$.
\begin{proof}
    Define the map 
    \begin{align*}
        \phi:[0,1]&\rightarrow \mathbb{R} \\
         x&\longmapsto \displaystyle\int_{C'\cap \mathbb{R}\times [x_0,x]\times \To^2} \dd \lambda_a
    \end{align*}
    Notice that by properties of $J$-holomorphic curves the function $\phi$ is always non-negative. By Stokes' theorem we have
    \begin{align*}
        \phi(x)=&\displaystyle\int_{C\cap [0,x]} \dd \lambda_a \\
               =&\displaystyle\int_{C_x}\lambda_a-\displaystyle\int_{C_{x_0}}\lambda_a\\
               =&\langle\lambda_a|_{x}-\lambda_a|_{x_0},[C_{x_0}]\rangle
    \end{align*}
    By taking the derivative of $\phi$ we found that 
    $$\phi'(x)=a'(x)\times [C_{x_0}]$$
    By continuity we can take $x\rightarrow x_0$ to obtain the result. 
    
    Now we prove that the equality in the inequality \eqref{positivity} implies that $C_x=\emptyset$. Suppose that $\phi'(x)=0$, since $\phi$ is positive we must have a local minimum at $x$ which implies that $\phi'(x)=0$, but 
    $$\phi''(x)=a''(x)\times [C_{x_0}]$$
    since $a''(x)$ never vanish and we have $\phi'(x)=\phi''(x)=0$ we can conclude that $[C_{x_0}]=0$.    
\end{proof}

Now we would like to have a more explicit version of $[C_{x_0}]$. Given a pair of orbit sets $\alpha=\{(\alpha_i,m_i)\}$ and $\beta=\{(\beta_j,n_j)\}$ and suppose that there are not special orbits, that is, orbits at the axis. Remember that for each $\alpha_i$ there exist unique $x_i^\alpha\in [0,1]$ such that $\alpha_i$ appear in $\pi(a(x)\times \To^2)$, similarly, for each $\beta_j$ there is a unique $x_i^\alpha\in [0,1]$ such that $\beta_j$ appear in $\pi(a(x)\times \To^2)$. Write the homology of these orbit sets as $[\alpha_i]=(v_i^\alpha,w_i^\alpha)\in \mathbb{Z}^2$ and $[\beta_j]=(v_j^\beta,w_j^\beta)\in \mathbb{Z}^2$. Organize $\alpha$ and $\beta$ in such a way that $x_1^\alpha<\dots< x_M^\alpha$ and $x_1^\beta<\dots< x_N^\beta$ where $M=\sum_i m_i$ and $N=\sum_j n_j$. Define the \textit{slice class} $\sigma_{\alpha,\beta}:[0,1]\rightarrow \mathbb{Z}^2$ as:

\begin{equation}\sigma_{\alpha,\beta}(x)= -\displaystyle\sum_{x_i^\alpha<x}m_i(v_i^\alpha,v_i^\alpha)+ \displaystyle\sum_{x_i^\beta<x} n_j(v_j^\beta,v_j^\beta)+(-c_\alpha+c_\beta)(n,m)    
\end{equation}

Similar to \cite[lem 5.8]{cristofaro2020proof}  we can prove that $[C_x]=\sigma_{\alpha,\beta}(x)$. Which gives a completely combinatorial interpretation to the equation \eqref{positivity}.

\subsection{Paths can not cross}
\label{Pathcannotcross}

An interesting consequence of positivity is that forced a certain order relation over the generators. More precisely we have the following lemma

\begin{lemm}
    If there exist a $J$-holomorphic curve $C$ from $\alpha$ to $\beta$, then $P_\beta$ is never above $P_\alpha$.
\end{lemm}


\begin{proof}
    By contradiction suppose that $P_\beta$ goes above $P_\alpha$, there must exist two intersection points which we call $(a,b)$ and $(c,d)$, with $a<c$. Then on the interval $(a,c)$ the path $P_\beta$ is strictly above $P_\alpha$ except at end points where they overlap. Form the line connecting $(a,c)$ and $(b,d)$, we can find $x_0\in (a,c)$ such that $f'(x_0)=\frac{d-b}{c-a}$. We compute $[C_{x_0-\epsilon}]$ and apply the equation \eqref{positivity}. 

    Let the lattice point $(p,q)$ have the following property: it is a vertex on $P_\alpha$, the edge to the left of this lattice point has slope less than $f'(x_0)$, and the edge to the right of this vertex has slope greater than equal to $f'(x_0)$. Then the contribution to $[C_{x_0-\epsilon}]$ from $P_\alpha$ is simply $(-(B-q),-p)$ where $B$ is the horizontal distance. We also consider the contribution of $C_{x_0-\epsilon}$ from $P_\beta$, which takes the form $(B-q',p')$. The lattice point $(p',q')$ on $P_\beta$ is chosen the same way as $(p,q)$. If no such vertex exists, then $P_\beta$ must overlap with the line segment connecting $(a,b)$ and $(c,d)$. Then the point $(p',q')$ is still the lattice point on $P_\beta$ which corresponds to the left most end point of where $P_\beta$ overlaps with the line connecting $(a,b)$ to $(c,d)$. In either case the positivity says that 
    $$(q-q')+\displaystyle\frac{d-b}{c-a}(p'-p)\geq 0$$
    We first assume $(p',q')$ is not on the line connecting $(a,b)$ and $(c,d)$, then this means that the point $(p,q)$ is futher away from the line connecting $(a,b)$ to $(c,d)$ than $(p',q')$. Geometrically this is described by 
    $$(b-d)(p-p')+(c-a)(q-q')<0$$
    which is impossible. Now assume $(p',q')$ is on the line connecting $(a,b)$ to $(c,d)$, then since we have chosen $[C_{x_0-\epsilon}]$, we must have $p'<p$. The energy inequality implies 
    $$\displaystyle\frac{q-q'}{p-p'}>\displaystyle\frac{d-b}{c-a}$$
    contradicting the geometric picture. 
\end{proof}
\subsection{Curves correspond to Corounding}
Using the result from the former sections we can prove the `only if part' of the Proposition \ref{equicomgeo} part 3. Notice that this proof is similar to \cite[Lem.~ 5.10]{cristofaro2020proof}.
More precisely we will prove the following lemma.

\begin{lemm}
    Suppose that $\lambda_a^\epsilon$ is a good perturbation of $\lambda_a$. Let $\alpha$ and $\beta$ be admissible generators such that $I(\alpha,\beta)=1$. Then, for a generic admissible $J$ close to $J_\text{std}$, 
    $$\langle\partial \alpha,\beta\rangle=1$$
    only if $\alpha$ is obtained by corounding the corner. 
\end{lemm}
\begin{proof}
    Suppose that 
    $$\langle\partial \alpha,\beta\rangle=1$$
    for some generically chosen $J$. We first choose $J$ generically to rule out double rounding, which we can do by the argument \cite[Lem.~A.1]{hutchings2005periodic}. 
    By positivity we now that $P_\alpha$ is above $P_\beta$. Consider the region between $P_\alpha$ and $P_\beta$. We can take this region and decompose it into two kinds of subregions: \textit{non-trivial} subregions where $P_\alpha$ is above $P_\beta$- meaning that the parts of $P_\alpha$ and $P_\beta$ intersect at most at two points in these regions, and, \textit{trivial} subregions where the concave paths (without the labels) coincide. 
    
    We begin by showing that there is at least one non-trivial region. Let suppose that this is not the case, this implies that $P_\alpha$ and $P_\beta$ coincide as \textit{unlabeled concave paths}. Let $C$ be the unique embedded component of a given $J$-holomorphic curve from $\alpha$ to $\beta$. From the Fredholm index it is easy to deduce that $\alpha$ is an elliptic orbit and $\beta$ is the corresponding hyperbolic orbit. Then it is possible to prove that these $J$-holomorphic to $J$-holomorphic curves obtain from the Morse-Bott perturbation as explain in \cite[Lem.~3.14]{hutchings2005periodic}. These $J$-holomorphic curves appear in pairs so their mod $2$ count vanishes. We conclude that there is at least one non-trivial region. 

    To symplify the proof, notice that it is enough to suppose that $C$ is irreducible. Indeed, if $C$ is not irreducible, consider its embedded component $C'$. Then it follows from Proposition \ref{BasedDiff} that there exist generators $\alpha'$ and $\beta'$ such that $P_\beta$ is obtained by corounding the corner from $P_\alpha$ if and only if $P_\beta'$ is obtained by corounding the corner from $P_\alpha'$. So we suppose from now on that $C$ is irreducible. 

    Since we already prove that there is at least one non-trivial region, we want two prove that under the assumption that $C$ is irreducibe, the polygonal paths $P_\alpha$ and $P_\beta$ form exactly one non-trivial region. From the Fredholm index we can deduce that there is at most one non-trivial region. Since we already prove that there is at least one non-trivial region we conclude that there is exactly one non-trivial region. 

    Now we argue that the number of trivial region is zero. First we want to argue that the region between $P_\alpha$ and $P_\beta$ cotain no interior points. In fact this follow from the equation $I(\alpha,\beta)=I(\alpha)-I(\beta)=2i_{\alpha,\beta}+2e_\beta+h+2(c_1+c_2)$ were $i_{\alpha,\beta}$ is the number of points in the region between $P_\alpha$ and $P_\beta$ that we can deduce from the Equation \eqref{echinccon}. Each case force us to conclude that $i_{\alpha,\beta}=0$. 

    Since each trivial part of the region between $P_\alpha$ and $P_\beta$ contribute to the Fredholm index we deduce that there is at most one trivial reigion. Let suppose that this trivial region is at the right of the non-trivial region, the other case is similar. Let $v_{p,q}$ be the vector that represent this trivial region. Notice that by the concavity, the edge in $P_\alpha$ inmediately to the left of $v_{p,q}$ correspond to a vector $v_{p',q'}$ such that $q/p<q'/p'$. Then we can use the \textit{slice class} of $C$ to find a contradiction. 

    To finish the proof it is enough to check the part corresponding to \textit{locally losing an h} also hold. Write $3=2e_\beta+h+2(c_1+c_2)$. Let $m_\beta$ be the total multiplicity of $\beta$. Notice that $2\leq m_\beta\leq 3$. It is easy to check that $m_\beta=2$ correspond exactly to the decoration define for coroounding the corner, While $m_\beta=3$ correspond to double rounding which we have chosen a perturbation that rule this possibility out. 

    This finish the proof. 
    \end{proof}

\subsection{Corounding correspond to curves}
\label{Coroundingcorrespondtocurves}
In this section we prove that the operation of corounding the corner over a pair of polygonal paths corresponds to a $J$-holomorphic curve between the corresponding Reeb orbit sets. 

We begin by proving that there are no J-holomorphic curves with genus bigger than zero in our case of interest.  
\begin{lemm}
\label{genuszero}
    Let $\alpha$ and $\beta$ be Reeb orbit sets with $I(\alpha,\beta)=1$ and no special orbits. Suppose that $C\in \mathcal{M}_1(\alpha,\beta)$ is irreducible then $C$ has genus zero. 
\end{lemm}
\begin{proof}
We prove this by contradiction. Suppose that $g>0$. 

From the Fredholm equation \ref{FredInd} we deduce that 
$$3=2g+2e_\beta+h+2 c_1+2 c_2$$
From the intersection of $C$ with the special orbits we deduce that $c_1$ and $c_2$ are positive numbers. Analizing the  equation we only can have  $c_1+c_2=1$ or $c_1=c_2=0$.

If $c_1=1$ than $\beta=\emptyset$ this implies that $I(\alpha)=1$. Therefore $I(\alpha)=2\mathcal{L}_{n,m}(\alpha)+h_\alpha=1$ but this equation contradicts that $g>1$. The case $c_2=1$ is similar. 

Now we can suppose that $c_1=c_2=0$. To prove this we use the no-crossing property of the previous section. We show that there cannot be a genus one curve satisfying the assumtions of the previous step. The Fredholm index formula tell us that 
$$1=h+2e_\beta$$
which means $e_\beta=0$ and $h_\alpha+h_\beta=1$. If $h_\alpha=1$ and $h_\beta=0$ then $\alpha_\beta$. By inspection $C$ cannot have $\text{ECH}$ index one. 
On the other hand, if $h_\alpha=0$ and $h_\beta=1$ then $P_\beta$ consists of a single line segment. By positivity and the fact that $P_\alpha$ has the same initial and ending points of $P_\beta$, we only can have that $P_\alpha=P_\beta$ as paths. It is easy to see that the ECH index can not be one. 
\end{proof}

The next important part of the proof is to show that irreducible curves correspond to non-trivial regions. 

\begin{lemm}
    \label{trivcil}
    Suppose that $\alpha$ and $\beta$ are admissible orbits such that $I(\alpha,\beta)=1$ and no special orbits. Write $\alpha=\gamma_1\alpha'\gamma_2$ and $\beta=\gamma_1\beta'\gamma_2$ such that $\alpha'$ and $\beta'$ does not have generators in common. Then there is a bijection 
    $$\mathcal{M}_1(\alpha,\beta)\cong \mathcal{M}_1(\alpha',\beta')$$
    given by attaching trivial cylinder to $\gamma_1$ and $\gamma_2$.    
\end{lemm}

\begin{proof}
    Notice that $I(\alpha,\beta)=I(\alpha',\beta')$. From intersection positivity and inequality \ref{intcyl} it follows that if $I(C')=I(C'\cup T)$ then $C'\cap T=\emptyset$. From this it follows that the map $\mathcal{M}(\alpha',\beta')\rightarrow\mathcal{M}(\alpha,\beta)$ it is well-defined and clearly injective. 

    To prove that this map is surjective we argue by contradiction. Suppose that $\mathcal{C}\in\mathcal{M}(\alpha,\beta)$ is such that does not contain trivial cylinders over all $\gamma_1\cup\gamma_2$. Let $C'$ denote the non-trivial component of $C$. Then $C'$ has both a positive end and a negative end at some orbit $\rho$ of $\gamma_1\cup\gamma_2$. 
    Notice that by Lemma \ref{genuszero} we can supose $g=0$ in Equation \ref{FredInd}. Lets suppose first that both relative Chern classes in Equation \ref{FredInd} are equal to zero. In this case, we can conclude that $\beta'$ is exactly one hyperbolic Reeb orbit and $\alpha'$ consist uniquely of elliptic orbits. This contradicts positivity.
    Now suppose that exactly one of the relative Chern classes in Equation \ref{FredInd} is equal to one. No matter if $\rho$ is hyperbolic or elliptic, since it is in both ends we conclude that $\text{ind}(C')=-1$ which is a contradiction. This proves surjectivity. 
\end{proof}

With these two lemmas at hand we can finally prove the following proposition. 

\begin{pro}
    Let $\alpha$ and $\beta$ admissible generators without special orbits such that $P_\beta$ is obtained from $P_\alpha$ by corounding the corner. Then for a generic almost complex structure $\#\mathcal{M}_1(\alpha,\beta)/\mathbb{R}=1$.
\end{pro}
    
\begin{proof}
     Notice that by Lemma \ref{trivcil} we can suppose that $\alpha$ and $\beta$ has no orbits in common. Since there are no special orbits for $\alpha$ or $\beta$ by positivity there exist a $x_1$ and $x_2$ such that any $J$-holormophic curve  connecting $\alpha$ with $\beta$ must be contained in $[x_1,x_2]\times \To^2$. After this observation the proof follows the same steps as in \cite[Lem.~3.17]{hutchings2005periodic} where the work of Taubes \cite{taubes2002compendium} is used to prove the existence of the corresponding J-holomorphic curve. 
\end{proof}

\printbibliography

\end{document}